\pdfoutput=1
\RequirePackage{ifpdf}
\ifpdf 
\documentclass[pdftex]{sigma}
\else
\documentclass{sigma}
\fi

\usepackage{bbm}
\usepackage[all]{xy}
\usepackage{mathrsfs}
\usepackage{enumitem}

\newtheorem{thm}{Theorem}[section]
\newtheorem{cor}[thm]{Corollary}
\newtheorem{prop}[thm]{Proposition}
\newtheorem{lem}[thm]{Lemma}
{ \theoremstyle{definition}
\newtheorem{Remark}[thm]{Remark}
\newtheorem{ex}[thm]{Example}
\newtheorem{Definition}[thm]{Definition}
}

\newcommand{\actionG}{\underline{\delta}}
\newcommand{\brep}{u}
\newcommand{\C}{\mathbbm{C}}
\newcommand{\CAlg}{\mathcal{O}}
\newcommand{\CatInd}[1][\brep]{\mathcal{T}_{#1}}

\newcommand{\ChainGp}{\mathfrak{C}}

\newcommand{\CQG}{\mathbbm{G}}

\newcommand{\Falg}[1][\alpha]{{\mathcal{O}^{#1}}}
\newcommand{\Hh}{\mathscr{H}}

\newcommand{\inj}{\hookrightarrow}
\newcommand{\InvJ}{\mathcal{F}}

\newcommand{\N}{\mathbbm{N}}

\newcommand{\Ss}{\mathcal{S}}

\newcommand{\TrivRep}{\epsilon}
\newcommand{\UCTclass}{\mathscr{N}}

\newcommand{\Z}{\mathbbm{Z}}

\DeclareMathOperator{\can}{can}

\DeclareMathOperator{\Hom}{Hom}
\DeclareMathOperator{\id}{Id}

\numberwithin{equation}{section}

\usepackage{calc}
\newcounter{PartitionDepth}
\newcounter{PartitionLength}

\newsavebox{\boxpaarpart}
 \savebox{\boxpaarpart}
 { \begin{picture}(0.7,0.5)
 \put(0,0){\line(0,1){0.4}}
 \put(0.5,0){\line(0,1){0.4}}
 \put(0,0.4){\line(1,0){0.5}}
 \end{picture}}
\newsavebox{\boxbaarpart}
 \savebox{\boxbaarpart}
 { \begin{picture}(0.7,0.5)
 \put(0,0){\line(0,1){0.4}}
 \put(0.5,0){\line(0,1){0.4}}
 \put(0,0){\line(1,0){0.5}}
 \end{picture}}
\newsavebox{\boxdreipart}
 \savebox{\boxdreipart}
 { \begin{picture}(1,0.5)
 \put(0,0){\line(0,1){0.4}}
 \put(0.4,0){\line(0,1){0.4}}
 \put(0.8,0){\line(0,1){0.4}}
 \put(0,0.4){\line(1,0){0.8}}
 \end{picture}}
\newsavebox{\boxvierpart}
 \savebox{\boxvierpart}
 { \begin{picture}(1.4,0.5)
 \put(0,0){\line(0,1){0.4}}
 \put(0.4,0){\line(0,1){0.4}}
 \put(0.8,0){\line(0,1){0.4}}
 \put(1.2,0){\line(0,1){0.4}}
 \put(0,0.4){\line(1,0){1.2}}
 \end{picture}}
\newsavebox{\boxvierpartrot}
 \savebox{\boxvierpartrot}
 { \begin{picture}(0.5,0.8)
 \put(0,-0.2){\line(0,1){0.3}}
 \put(0.4,-0.2){\line(0,1){0.3}}
 \put(0,0.1){\line(1,0){0.4}}
 \put(0,0.4){\line(0,1){0.3}}
 \put(0.4,0.4){\line(0,1){0.3}}
 \put(0,0.4){\line(1,0){0.4}}
 \put(0.2,0.1){\line(0,1){0.3}}
 \end{picture}}
\newsavebox{\boxvierpartrotdrei}
 \savebox{\boxvierpartrotdrei}
 { \begin{picture}(1,0.8)
 \put(0,-0.2){\line(0,1){0.3}}
 \put(0.4,-0.2){\line(0,1){0.8}}
 \put(0.8,-0.2){\line(0,1){0.3}}
 \put(0,0.1){\line(1,0){0.8}}
 \end{picture}}
\newsavebox{\boxcrosspart}
 \savebox{\boxcrosspart}
 { \begin{picture}(0.5,0.8)
 \put(0,-0.2){\line(1,2){0.4}}
 \put(0.4,-0.2){\line(-1,2){0.4}}
 \end{picture}}
\newsavebox{\boxhalflibpart}
 \savebox{\boxhalflibpart}
 { \begin{picture}(1,0.8)
 \put(0,-0.2){\line(1,1){0.8}}
 \put(0.8,-0.2){\line(-1,1){0.8}}
 \put(0.4,-0.2){\line(0,1){0.8}}
 \end{picture}}
\newsavebox{\boxpositioner}
 \savebox{\boxpositioner}
 { \begin{picture}(1.4,0.5)
 \put(0,0){\line(0,1){0.3}}
 \put(0.4,0){\line(0,1){0.6}}
 \put(0.8,0){\line(0,1){0.3}}
 \put(1.2,0){\line(0,1){0.6}}
 \put(0.4,0.6){\line(1,0){0.8}}
 \end{picture}}
\newsavebox{\boxfatcross}
 \savebox{\boxfatcross}
 { \begin{picture}(1.4,1.3)
 \put(0,-0.2){\line(0,1){0.3}}
 \put(0.4,-0.2){\line(0,1){0.3}}
 \put(0.8,-0.2){\line(0,1){0.3}}
 \put(1.2,-0.2){\line(0,1){0.3}}
 \put(0,0.1){\line(1,0){0.4}}
 \put(0.8,0.1){\line(1,0){0.4}}
 \put(0,0.9){\line(0,1){0.3}}
 \put(0.4,0.9){\line(0,1){0.3}}
 \put(0.8,0.9){\line(0,1){0.3}}
 \put(1.2,0.9){\line(0,1){0.3}}
 \put(0,0.9){\line(1,0){0.4}}
 \put(0.8,0.9){\line(1,0){0.4}}
 \put(0.2,0.1){\line(1,1){0.8}}
 \put(0.2,0.9){\line(1,-1){0.8}}
 \end{picture}}
\newsavebox{\boxprimarypart}
 \savebox{\boxprimarypart}
 { \begin{picture}(1,1.3)
 \put(0,-0.2){\line(0,1){0.3}}
 \put(0.4,-0.2){\line(0,1){0.3}}
 \put(0.8,-0.2){\line(0,1){0.3}}
 \put(0.4,0.1){\line(1,0){0.4}}
 \put(0,0.9){\line(0,1){0.3}}
 \put(0.4,0.9){\line(0,1){0.3}}
 \put(0.8,0.9){\line(0,1){0.3}}
 \put(0,0.9){\line(1,0){0.4}}
 \put(0,0.1){\line(1,1){0.8}}
 \put(0.2,0.9){\line(1,-2){0.4}}
 \end{picture}}

\newcommand{\twocol}{\circ\bullet}

\newsavebox{\boxidpartww}
 \savebox{\boxidpartww}
 { \begin{picture}(0.5,1.2)
 \put(0.2,0.15){\line(0,1){0.7}}
 \put(0,-0.2){$\circ$}
 \put(0,0.8){$\circ$}
 \end{picture}}
\newsavebox{\boxidpartbw}
 \savebox{\boxidpartbw}
 { \begin{picture}(0.5,1.2)
 \put(0.2,0.15){\line(0,1){0.7}}
 \put(0,-0.2){$\circ$}
 \put(0,0.8){$\bullet$}
 \end{picture}}
\newsavebox{\boxidpartwb}
 \savebox{\boxidpartwb}
 { \begin{picture}(0.5,1.2)
 \put(0.2,0.15){\line(0,1){0.7}}
 \put(0,-0.2){$\bullet$}
 \put(0,0.8){$\circ$}
 \end{picture}}
\newsavebox{\boxidpartbb}
 \savebox{\boxidpartbb}
 { \begin{picture}(0.5,1.2)
 \put(0.2,0.15){\line(0,1){0.7}}
 \put(0,-0.2){$\bullet$}
 \put(0,0.8){$\bullet$}
 \end{picture}}

\newsavebox{\boxidpartsingletonbb}
 \savebox{\boxidpartsingletonbb}
 { \begin{picture}(0.5,1.2)
 \put(0.2,0.15){\line(0,1){0.2}}
 \put(0.2,0.65){\line(0,1){0.2}}
 \put(0,-0.2){$\bullet$}
 \put(0,0.8){$\bullet$}
 \end{picture}}
\newsavebox{\boxidpartsingletonww}
 \savebox{\boxidpartsingletonww}
 { \begin{picture}(0.5,1.2)
 \put(0.2,0.15){\line(0,1){0.2}}
 \put(0.2,0.65){\line(0,1){0.2}}
 \put(0,-0.2){$\circ$}
 \put(0,0.8){$\circ$}
 \end{picture}}

\newsavebox{\boxpaarpartbb}
 \savebox{\boxpaarpartbb}
 { \begin{picture}(1,1)
 \put(0.2,0.2){\line(0,1){0.4}}
 \put(0.7,0.2){\line(0,1){0.4}}
 \put(0.2,0.6){\line(1,0){0.5}}
 \put(0.5,-0.2){$\bullet$}
 \put(0,-0.2){$\bullet$}
 \end{picture}}
\newsavebox{\boxpaarpartww}
 \savebox{\boxpaarpartww}
 { \begin{picture}(1,1)
 \put(0.2,0.2){\line(0,1){0.4}}
 \put(0.7,0.2){\line(0,1){0.4}}
 \put(0.2,0.6){\line(1,0){0.5}}
 \put(0.5,-0.2){$\circ$}
 \put(0,-0.2){$\circ$}
 \end{picture}}
\newsavebox{\boxpaarpartbw}
 \savebox{\boxpaarpartbw}
 { \begin{picture}(1,1)
 \put(0.2,0.2){\line(0,1){0.4}}
 \put(0.7,0.2){\line(0,1){0.4}}
 \put(0.2,0.6){\line(1,0){0.5}}
 \put(0.5,-0.2){$\circ$}
 \put(0,-0.2){$\bullet$}
 \end{picture}}
\newsavebox{\boxpaarpartwb}
 \savebox{\boxpaarpartwb}
 { \begin{picture}(1,1)
 \put(0.2,0.2){\line(0,1){0.4}}
 \put(0.7,0.2){\line(0,1){0.4}}
 \put(0.2,0.6){\line(1,0){0.5}}
 \put(0.5,-0.2){$\bullet$}
 \put(0,-0.2){$\circ$}
 \end{picture}}
\newsavebox{\boxbaarpartbb}
 \savebox{\boxbaarpartbb}
 { \begin{picture}(1,1)
 \put(0.2,-0.1){\line(0,1){0.4}}
 \put(0.7,-0.1){\line(0,1){0.4}}
 \put(0.2,-0.1){\line(1,0){0.5}}
 \put(0.5,0.3){$\bullet$}
 \put(0,0.3){$\bullet$}
 \end{picture}}
\newsavebox{\boxbaarpartww}
 \savebox{\boxbaarpartww}
 { \begin{picture}(1,1)
 \put(0.2,-0.1){\line(0,1){0.4}}
 \put(0.7,-0.1){\line(0,1){0.4}}
 \put(0.2,-0.1){\line(1,0){0.5}}
 \put(0.5,0.3){$\circ$}
 \put(0,0.3){$\circ$}
 \end{picture}}
\newsavebox{\boxbaarpartbw}
 \savebox{\boxbaarpartbw}
 { \begin{picture}(1,1)
 \put(0.2,-0.1){\line(0,1){0.4}}
 \put(0.7,-0.1){\line(0,1){0.4}}
 \put(0.2,-0.1){\line(1,0){0.5}}
 \put(0.5,0.3){$\circ$}
 \put(0,0.3){$\bullet$}
 \end{picture}}
\newsavebox{\boxbaarpartwb}
 \savebox{\boxbaarpartwb}
 { \begin{picture}(1,1)
 \put(0.2,-0.1){\line(0,1){0.4}}
 \put(0.7,-0.1){\line(0,1){0.4}}
 \put(0.2,-0.1){\line(1,0){0.5}}
 \put(0.5,0.3){$\bullet$}
 \put(0,0.3){$\circ$}
 \end{picture}}

\newsavebox{\boxcutpaarpartbb}
 \savebox{\boxcutpaarpartbb}
 { \begin{picture}(1,1)
 \put(0,0.2){\line(0,1){0.4}}
 \put(0.8,0.2){\line(0,1){0.4}}
 \put(0,0.6){\line(1,0){0.2}}
 \put(0.6,0.6){\line(1,0){0.2}}
 \put(0.6,-0.2){$\bullet$}
 \put(-0.2,-0.2){$\bullet$}
 \end{picture}}
\newsavebox{\boxcutpaarpartww}
 \savebox{\boxcutpaarpartww}
 { \begin{picture}(1,1)
 \put(0,0.2){\line(0,1){0.4}}
 \put(0.8,0.2){\line(0,1){0.4}}
 \put(0,0.6){\line(1,0){0.2}}
 \put(0.6,0.6){\line(1,0){0.2}}
 \put(0.6,-0.2){$\circ$}
 \put(-0.2,-0.2){$\circ$}
 \end{picture}}
\newsavebox{\boxcutpaarpartbw}
 \savebox{\boxcutpaarpartbw}
 { \begin{picture}(1,1)
 \put(0,0.2){\line(0,1){0.4}}
 \put(0.8,0.2){\line(0,1){0.4}}
 \put(0,0.6){\line(1,0){0.2}}
 \put(0.6,0.6){\line(1,0){0.2}}
 \put(0.6,-0.2){$\circ$}
 \put(-0.2,-0.2){$\bullet$}
 \end{picture}}
\newsavebox{\boxcutpaarpartwb}
 \savebox{\boxcutpaarpartwb}
 { \begin{picture}(1,1)
 \put(0,0.2){\line(0,1){0.4}}
 \put(0.8,0.2){\line(0,1){0.4}}
 \put(0,0.6){\line(1,0){0.2}}
 \put(0.6,0.6){\line(1,0){0.2}}
 \put(0.6,-0.2){$\bullet$}
 \put(-0.2,-0.2){$\circ$}
 \end{picture}}

\newsavebox{\boxsingletonw}
 \savebox{\boxsingletonw}
 { \begin{picture}(0.5, 1)
 \put(0,0.3){$\uparrow$}
 \put(0,-0.2){$\circ$}
 \end{picture}}
\newsavebox{\boxsingletonb}
 \savebox{\boxsingletonb}
 { \begin{picture}(0.5, 1)
 \put(0,0.3){$\uparrow$}
 \put(0,-0.2){$\bullet$}
 \end{picture}}
\newsavebox{\boxdownsingletonw}
 \savebox{\boxdownsingletonw}
 { \begin{picture}(0.5, 1)
 \put(0,0){$\downarrow$}
 \put(0,0.6){$\circ$}
 \end{picture}}
\newsavebox{\boxdownsingletonb}
 \savebox{\boxdownsingletonb}
 { \begin{picture}(0.5, 1)
 \put(0,0){$\downarrow$}
 \put(0,0.6){$\bullet$}
 \end{picture}}

\newsavebox{\boxvierpartwbwb}
 \savebox{\boxvierpartwbwb}
 { \begin{picture}(1.7,1)
 \put(0.2,0.2){\line(0,1){0.4}}
 \put(0.6,0.2){\line(0,1){0.4}}
 \put(1,0.2){\line(0,1){0.4}}
 \put(1.4,0.2){\line(0,1){0.4}}
 \put(0.2,0.6){\line(1,0){1.2}}
 \put(0,-0.2){$\circ$}
 \put(0.4,-0.2){$\bullet$}
 \put(0.8,-0.2){$\circ$}
 \put(1.2,-0.2){$\bullet$}
 \end{picture}}
\newsavebox{\boxvierpartbwbw}
 \savebox{\boxvierpartbwbw}
 { \begin{picture}(1.7,1)
 \put(0.2,0.2){\line(0,1){0.4}}
 \put(0.6,0.2){\line(0,1){0.4}}
 \put(1,0.2){\line(0,1){0.4}}
 \put(1.4,0.2){\line(0,1){0.4}}
 \put(0.2,0.6){\line(1,0){1.2}}
 \put(0,-0.2){$\bullet$}
 \put(0.4,-0.2){$\circ$}
 \put(0.8,-0.2){$\bullet$}
 \put(1.2,-0.2){$\circ$}
 \end{picture}}
\newsavebox{\boxvierpartwwbb}
 \savebox{\boxvierpartwwbb}
 { \begin{picture}(1.7,1)
 \put(0.2,0.2){\line(0,1){0.4}}
 \put(0.6,0.2){\line(0,1){0.4}}
 \put(1,0.2){\line(0,1){0.4}}
 \put(1.4,0.2){\line(0,1){0.4}}
 \put(0.2,0.6){\line(1,0){1.2}}
 \put(0,-0.2){$\circ$}
 \put(0.4,-0.2){$\circ$}
 \put(0.8,-0.2){$\bullet$}
 \put(1.2,-0.2){$\bullet$}
 \end{picture}}
\newsavebox{\boxvierpartwbbw}
 \savebox{\boxvierpartwbbw}
 { \begin{picture}(1.7,1)
 \put(0.2,0.2){\line(0,1){0.4}}
 \put(0.6,0.2){\line(0,1){0.4}}
 \put(1,0.2){\line(0,1){0.4}}
 \put(1.4,0.2){\line(0,1){0.4}}
 \put(0.2,0.6){\line(1,0){1.2}}
 \put(0,-0.2){$\circ$}
 \put(0.4,-0.2){$\bullet$}
 \put(0.8,-0.2){$\bullet$}
 \put(1.2,-0.2){$\circ$}
 \end{picture}}
\newsavebox{\boxvierpartbwwb}
 \savebox{\boxvierpartbwwb}
 { \begin{picture}(1.7,1)
 \put(0.2,0.2){\line(0,1){0.4}}
 \put(0.6,0.2){\line(0,1){0.4}}
 \put(1,0.2){\line(0,1){0.4}}
 \put(1.4,0.2){\line(0,1){0.4}}
 \put(0.2,0.6){\line(1,0){1.2}}
 \put(0,-0.2){$\bullet$}
 \put(0.4,-0.2){$\circ$}
 \put(0.8,-0.2){$\circ$}
 \put(1.2,-0.2){$\bullet$}
 \end{picture}}
\newsavebox{\boxvierpartrotwbwb}
 \savebox{\boxvierpartrotwbwb}
 { \begin{picture}(1,1.2)
 \put(0.2,0.15){\line(0,1){0.2}}
 \put(0.2,0.65){\line(0,1){0.2}}
 \put(0.2,0.65){\line(1,0){0.5}}
 \put(0.2,0.35){\line(1,0){0.5}}
 \put(0.45,0.35){\line(0,1){0.3}}
 \put(0.7,0.15){\line(0,1){0.2}}
 \put(0.7,0.65){\line(0,1){0.2}}
 \put(0,0.8){$\circ$}
 \put(0.5,0.8){$\bullet$}
 \put(0,-0.2){$\circ$}
 \put(0.5,-0.2){$\bullet$}
 \end{picture}}
\newsavebox{\boxvierpartrotbwwb}
 \savebox{\boxvierpartrotbwwb}
 { \begin{picture}(1,1.2)
 \put(0.2,0.15){\line(0,1){0.2}}
 \put(0.2,0.65){\line(0,1){0.2}}
 \put(0.2,0.65){\line(1,0){0.5}}
 \put(0.2,0.35){\line(1,0){0.5}}
 \put(0.45,0.35){\line(0,1){0.3}}
 \put(0.7,0.15){\line(0,1){0.2}}
 \put(0.7,0.65){\line(0,1){0.2}}
 \put(0,0.8){$\bullet$}
 \put(0.5,0.8){$\circ$}
 \put(0,-0.2){$\circ$}
 \put(0.5,-0.2){$\bullet$}
 \end{picture}}
\newsavebox{\boxvierpartrotbwbw}
 \savebox{\boxvierpartrotbwbw}
 { \begin{picture}(1,1.2)
 \put(0.2,0.15){\line(0,1){0.2}}
 \put(0.2,0.65){\line(0,1){0.2}}
 \put(0.2,0.65){\line(1,0){0.5}}
 \put(0.2,0.35){\line(1,0){0.5}}
 \put(0.45,0.35){\line(0,1){0.3}}
 \put(0.7,0.15){\line(0,1){0.2}}
 \put(0.7,0.65){\line(0,1){0.2}}
 \put(0,0.8){$\bullet$}
 \put(0.5,0.8){$\circ$}
 \put(0,-0.2){$\bullet$}
 \put(0.5,-0.2){$\circ$}
 \end{picture}}
\newsavebox{\boxvierpartrotwwww}
 \savebox{\boxvierpartrotwwww}
 { \begin{picture}(1,1.2)
 \put(0.2,0.15){\line(0,1){0.2}}
 \put(0.2,0.65){\line(0,1){0.2}}
 \put(0.2,0.65){\line(1,0){0.5}}
 \put(0.2,0.35){\line(1,0){0.5}}
 \put(0.45,0.35){\line(0,1){0.3}}
 \put(0.7,0.15){\line(0,1){0.2}}
 \put(0.7,0.65){\line(0,1){0.2}}
 \put(0,0.8){$\circ$}
 \put(0.5,0.8){$\circ$}
 \put(0,-0.2){$\circ$}
 \put(0.5,-0.2){$\circ$}
 \end{picture}}
\newsavebox{\boxvierpartrotbbbb}
 \savebox{\boxvierpartrotbbbb}
 { \begin{picture}(1,1.2)
 \put(0.2,0.15){\line(0,1){0.2}}
 \put(0.2,0.65){\line(0,1){0.2}}
 \put(0.2,0.65){\line(1,0){0.5}}
 \put(0.2,0.35){\line(1,0){0.5}}
 \put(0.45,0.35){\line(0,1){0.3}}
 \put(0.7,0.15){\line(0,1){0.2}}
 \put(0.7,0.65){\line(0,1){0.2}}
 \put(0,0.8){$\bullet$}
 \put(0.5,0.8){$\bullet$}
 \put(0,-0.2){$\bullet$}
 \put(0.5,-0.2){$\bullet$}
 \end{picture}}

\newsavebox{\boxdreipartwww}
 \savebox{\boxdreipartwww}
 { \begin{picture}(1.2,1)
 \put(0.2,0.2){\line(0,1){0.4}}
 \put(0.6,0.2){\line(0,1){0.4}}
 \put(1,0.2){\line(0,1){0.4}}
 \put(0.2,0.6){\line(1,0){0.8}}
 \put(0,-0.2){$\circ$}
 \put(0.4,-0.2){$\circ$}
 \put(0.8,-0.2){$\circ$}
 \end{picture}}

\newsavebox{\boxsechspartwbwbwb}
 \savebox{\boxsechspartwbwbwb}
 { \begin{picture}(2.5,1)
 \put(0.2,0.2){\line(0,1){0.4}}
 \put(0.6,0.2){\line(0,1){0.4}}
 \put(1,0.2){\line(0,1){0.4}}
 \put(1.4,0.2){\line(0,1){0.4}}
 \put(1.8,0.2){\line(0,1){0.4}}
 \put(2.2,0.2){\line(0,1){0.4}}
 \put(0.2,0.6){\line(1,0){2}}
 \put(0,-0.2){$\circ$}
 \put(0.4,-0.2){$\bullet$}
 \put(0.8,-0.2){$\circ$}
 \put(1.2,-0.2){$\bullet$}
 \put(1.6,-0.2){$\circ$}
 \put(2.0,-0.2){$\bullet$}
 \end{picture}}

\newsavebox{\boxcrosspartwbbw}
 \savebox{\boxcrosspartwbbw}
 { \begin{picture}(1,1.4)
 \put(0.25,0.15){\line(1,2){0.4}}
 \put(0.65,0.15){\line(-1,2){0.4}}
 \put(0,0.9){$\circ$}
 \put(0.5,0.9){$\bullet$}
 \put(0,-0.2){$\bullet$}
 \put(0.5,-0.2){$\circ$}
 \end{picture}}
\newsavebox{\boxcrosspartbwwb}
 \savebox{\boxcrosspartbwwb}
 { \begin{picture}(1,1.4)
 \put(0.25,0.15){\line(1,2){0.4}}
 \put(0.65,0.15){\line(-1,2){0.4}}
 \put(0,0.9){$\bullet$}
 \put(0.5,0.9){$\circ$}
 \put(0,-0.2){$\circ$}
 \put(0.5,-0.2){$\bullet$}
 \end{picture}}
\newsavebox{\boxcrosspartwwww}
 \savebox{\boxcrosspartwwww}
 { \begin{picture}(1,1.4)
 \put(0.25,0.15){\line(1,2){0.4}}
 \put(0.65,0.15){\line(-1,2){0.4}}
 \put(0,0.9){$\circ$}
 \put(0.5,0.9){$\circ$}
 \put(0,-0.2){$\circ$}
 \put(0.5,-0.2){$\circ$}
 \end{picture}}
\newsavebox{\boxcrosspartbbbb}
 \savebox{\boxcrosspartbbbb}
 { \begin{picture}(1,1.4)
 \put(0.25,0.15){\line(1,2){0.4}}
 \put(0.65,0.15){\line(-1,2){0.4}}
 \put(0,0.9){$\bullet$}
 \put(0.5,0.9){$\bullet$}
 \put(0,-0.2){$\bullet$}
 \put(0.5,-0.2){$\bullet$}
 \end{picture}}

\newsavebox{\boxhalflibpartwwwwww}
 \savebox{\boxhalflibpartwwwwww}
 { \begin{picture}(1.5,1.4)
 \put(0.3,0.15){\line(1,1){0.8}}
 \put(1.1,0.15){\line(-1,1){0.8}}
 \put(0.7,0.15){\line(0,1){0.8}}
 \put(0,0.9){$\circ$}
 \put(0.5,0.9){$\circ$}
 \put(1,0.9){$\circ$}
 \put(0,-0.2){$\circ$}
 \put(0.5,-0.2){$\circ$}
 \put(1,-0.2){$\circ$}
 \end{picture}}

\newsavebox{\boxpositionerd}
 \savebox{\boxpositionerd}
 { \begin{picture}(2.6,1.5)
 \put(0.9,0.2){\line(0,1){0.9}}
 \put(2.3,0.2){\line(0,1){0.9}}
 \put(0.9,1.1){\line(1,0){1.4}}
 \put(0.05,-0.05){$^\uparrow$}
 \put(0.2,0.4){\tiny$^{\otimes d}$}
 \put(1.35,-0.05){$^\uparrow$}
 \put(1.5,0.4){\tiny$^{\otimes d}$}
 \put(0,-0.2){$\circ$}
 \put(0.7,-0.2){$\circ$}
 \put(1.3,-0.2){$\bullet$}
 \put(2.1,-0.2){$\bullet$}
 \end{picture}}
\newsavebox{\boxpositionerrpluseins}
 \savebox{\boxpositionerrpluseins}
 { \begin{picture}(3.9,1.5)
 \put(1.6,0.2){\line(0,1){0.9}}
 \put(3.6,0.2){\line(0,1){0.9}}
 \put(1.6,1.1){\line(1,0){2}}
 \put(0.05,-0.05){$^\uparrow$}
 \put(0.2,0.4){\tiny$^{\otimes r+1}$}
 \put(2.05,-0.05){$^\uparrow$}
 \put(2.2,0.4){\tiny$^{\otimes r-1}$}
 \put(0,-0.2){$\circ$}
 \put(1.4,-0.2){$\bullet$}
 \put(2,-0.2){$\bullet$}
 \put(3.4,-0.2){$\bullet$}
 \end{picture}}
\newsavebox{\boxpositionerdt}
 \savebox{\boxpositionerdt}
 { \begin{picture}(2.9,1.5)
 \put(1.2,0.2){\line(0,1){0.9}}
 \put(2.9,0.2){\line(0,1){0.9}}
 \put(1.2,1.1){\line(1,0){1.7}}
 \put(0.05,-0.05){$^\uparrow$}
 \put(0.2,0.4){\tiny$^{\otimes dt}$}
 \put(1.65,-0.05){$^\uparrow$}
 \put(1.8,0.4){\tiny$^{\otimes dt}$}
 \put(0,-0.2){$\circ$}
 \put(1,-0.2){$\circ$}
 \put(1.6,-0.2){$\bullet$}
 \put(2.7,-0.2){$\bullet$}
 \end{picture}}
\newsavebox{\boxpositioners}
 \savebox{\boxpositioners}
 { \begin{picture}(2.6,1.5)
 \put(0.9,0.2){\line(0,1){0.9}}
 \put(2.3,0.2){\line(0,1){0.9}}
 \put(0.9,1.1){\line(1,0){1.4}}
 \put(0.05,-0.05){$^\uparrow$}
 \put(0.2,0.4){\tiny$^{\otimes s}$}
 \put(1.35,-0.05){$^\uparrow$}
 \put(1.5,0.4){\tiny$^{\otimes s}$}
 \put(0,-0.2){$\circ$}
 \put(0.7,-0.2){$\circ$}
 \put(1.3,-0.2){$\bullet$}
 \put(2.1,-0.2){$\bullet$}
 \end{picture}}
\newsavebox{\boxpositionerdinv}
 \savebox{\boxpositionerdinv}
 { \begin{picture}(2.6,1.5)
 \put(0.9,0.2){\line(0,1){0.9}}
 \put(2.3,0.2){\line(0,1){0.9}}
 \put(0.9,1.1){\line(1,0){1.4}}
 \put(0.05,-0.05){$^\uparrow$}
 \put(0.2,0.4){\tiny$^{\otimes d}$}
 \put(1.35,-0.05){$^\uparrow$}
 \put(1.5,0.4){\tiny$^{\otimes d}$}
 \put(0,-0.2){$\circ$}
 \put(0.7,-0.2){$\bullet$}
 \put(1.3,-0.2){$\bullet$}
 \put(2.1,-0.2){$\circ$}
 \end{picture}}
\newsavebox{\boxpositionersinv}
 \savebox{\boxpositionersinv}
 { \begin{picture}(2.6,1.5)
 \put(0.9,0.2){\line(0,1){0.9}}
 \put(2.3,0.2){\line(0,1){0.9}}
 \put(0.9,1.1){\line(1,0){1.4}}
 \put(0.05,-0.05){$^\uparrow$}
 \put(0.2,0.4){\tiny$^{\otimes s}$}
 \put(1.35,-0.05){$^\uparrow$}
 \put(1.5,0.4){\tiny$^{\otimes s}$}
 \put(0,-0.2){$\circ$}
 \put(0.7,-0.2){$\bullet$}
 \put(1.3,-0.2){$\bullet$}
 \put(2.1,-0.2){$\circ$}
 \end{picture}}
\newsavebox{\boxpositionerdpluszwei}
 \savebox{\boxpositionerdpluszwei}
 { \begin{picture}(3.4,1.5)
 \put(1.6,0.2){\line(0,1){0.9}}
 \put(3.0,0.2){\line(0,1){0.9}}
 \put(1.6,1.1){\line(1,0){1.4}}
 \put(0.05,-0.05){$^\uparrow$}
 \put(0.2,0.4){\tiny$^{\otimes d+2}$}
 \put(2.05,-0.05){$^\uparrow$}
 \put(2.2,0.4){\tiny$^{\otimes d}$}
 \put(0,-0.2){$\circ$}
 \put(1.4,-0.2){$\bullet$}
 \put(2,-0.2){$\bullet$}
 \put(2.8,-0.2){$\bullet$}
 \end{picture}}
\newsavebox{\boxpositionersminuszwei}
 \savebox{\boxpositionersminuszwei}
 { \begin{picture}(3.4,1.5)
 \put(1,0.2){\line(0,1){0.9}}
 \put(3.0,0.2){\line(0,1){0.9}}
 \put(1,1.1){\line(1,0){2}}
 \put(0.05,-0.05){$^\uparrow$}
 \put(0.2,0.4){\tiny$^{\otimes s}$}
 \put(1.45,-0.05){$^\uparrow$}
 \put(1.6,0.4){\tiny$^{\otimes s-2}$}
 \put(0,-0.2){$\circ$}
 \put(0.8,-0.2){$\bullet$}
 \put(1.4,-0.2){$\bullet$}
 \put(2.8,-0.2){$\bullet$}
 \end{picture}}
\newsavebox{\boxpositionerrnull}
 \savebox{\boxpositionerrnull}
 { \begin{picture}(3.8,1.5)
 \put(1.2,0.2){\line(0,1){0.9}}
 \put(3.4,0.2){\line(0,1){0.9}}
 \put(1.2,1.1){\line(1,0){2.2}}
 \put(0.05,-0.05){$^\uparrow$}
 \put(0.2,0.4){\tiny$^{\otimes r_0}$}
 \put(1.65,-0.05){$^\uparrow$}
 \put(1.8,0.4){\tiny$^{\otimes r_0-2}$}
 \put(0,-0.2){$\circ$}
 \put(1,-0.2){$\bullet$}
 \put(1.6,-0.2){$\bullet$}
 \put(3.2,-0.2){$\bullet$}
 \end{picture}}
\newsavebox{\boxpositionerwbwb}
 \savebox{\boxpositionerwbwb}
 { \begin{picture}(1.8,1)
 \put(0.6,0.2){\line(0,1){0.6}}
 \put(1.4,0.2){\line(0,1){0.6}}
 \put(0.6,0.8){\line(1,0){0.8}}
 \put(0.05,-0.05){$^\uparrow$}
 \put(0.85,-0.05){$^\uparrow$}
 \put(0,-0.2){$\circ$}
 \put(0.4,-0.2){$\bullet$}
 \put(0.8,-0.2){$\circ$}
 \put(1.2,-0.2){$\bullet$}
 \end{picture}}
\newsavebox{\boxpositionerwwbb}
 \savebox{\boxpositionerwwbb}
 { \begin{picture}(1.8,1)
 \put(0.6,0.2){\line(0,1){0.6}}
 \put(1.4,0.2){\line(0,1){0.6}}
 \put(0.6,0.8){\line(1,0){0.8}}
 \put(0.05,-0.05){$^\uparrow$}
 \put(0.85,-0.05){$^\uparrow$}
 \put(0,-0.2){$\circ$}
 \put(0.4,-0.2){$\circ$}
 \put(0.8,-0.2){$\bullet$}
 \put(1.2,-0.2){$\bullet$}
 \end{picture}}
\newsavebox{\boxpositionerrevwbwb}
 \savebox{\boxpositionerrevwbwb}
 { \begin{picture}(1.8,1)
 \put(0.2,0.2){\line(0,1){0.6}}
 \put(1.0,0.2){\line(0,1){0.6}}
 \put(0.2,0.8){\line(1,0){0.8}}
 \put(0.45,-0.05){$^\uparrow$}
 \put(1.25,-0.05){$^\uparrow$}
 \put(0,-0.2){$\circ$}
 \put(0.4,-0.2){$\bullet$}
 \put(0.8,-0.2){$\circ$}
 \put(1.2,-0.2){$\bullet$}
 \end{picture}}
\newsavebox{\boxpositionersalphaw}
 \savebox{\boxpositionersalphaw}
 { \begin{picture}(2.6,1.5)
 \put(0.9,0.2){\line(0,1){0.9}}
 \put(2.3,0.2){\line(0,1){0.9}}
 \put(0.9,1.1){\line(1,0){1.4}}
 \put(0.05,-0.05){$^\uparrow$}
 \put(1.35,-0.05){$^\uparrow$}
 \put(0.2,0.4){\tiny$^{\otimes s}$}
 \put(1.5,0.4){\tiny$^{\otimes \alpha}$}
 \put(0,-0.2){$\circ$}
 \put(0.7,-0.2){$\circ$}
 \put(1.3,-0.2){$\circ$}
 \put(2.1,-0.2){$\bullet$}
 \end{picture}}
\newsavebox{\boxpositionersalphab}
 \savebox{\boxpositionersalphab}
 { \begin{picture}(2.6,1.5)
 \put(0.9,0.2){\line(0,1){0.9}}
 \put(2.3,0.2){\line(0,1){0.9}}
 \put(0.9,1.1){\line(1,0){1.4}}
 \put(0.05,-0.05){$^\uparrow$}
 \put(1.35,-0.05){$^\uparrow$}
 \put(0.2,0.4){\tiny$^{\otimes s}$}
 \put(1.5,0.4){\tiny$^{\otimes \alpha}$}
 \put(0,-0.2){$\circ$}
 \put(0.7,-0.2){$\circ$}
 \put(1.3,-0.2){$\bullet$}
 \put(2.1,-0.2){$\bullet$}
 \end{picture}}

\newsavebox{\boxAspace}
 \savebox{\boxAspace}
 { \begin{picture}(0.1,1.5)
 \end{picture}}
\newsavebox{\boxBspace}
 \savebox{\boxBspace}
 { \begin{picture}(0,1.85)
 \end{picture}}

\newcommand{\idpartww}{\usebox{\boxidpartww}}
\newcommand{\idpartwb}{\usebox{\boxidpartwb}}

\newcommand{\idpartbb}{\usebox{\boxidpartbb}}

\newcommand{\paarpartww}{\usebox{\boxpaarpartww}}
\newcommand{\paarpartbw}{\usebox{\boxpaarpartbw}}
\newcommand{\paarpartwb}{\usebox{\boxpaarpartwb}}
\newcommand{\paarpartbb}{\usebox{\boxpaarpartbb}}
\newcommand{\baarpartww}{\usebox{\boxbaarpartww}}

\newcommand{\baarpartwb}{\usebox{\boxbaarpartwb}}

\newcommand{\singletonw}{\usebox{\boxsingletonw}}
\newcommand{\singletonb}{\usebox{\boxsingletonb}}

\newcommand{\vierpartwwbb}{\usebox{\boxvierpartwwbb}}

\newcommand{\vierpartrotwwww}{\usebox{\boxvierpartrotwwww}}

\newcommand{\CC}{\mathcal C}

\DeclareMathOperator{\Proj}{Proj}
\DeclareMathOperator{\white}{white}
\DeclareMathOperator{\nest}{nest}

\begin{document}


\newcommand{\arXivNumber}{1606.00569}

\renewcommand{\PaperNumber}{097}

\FirstPageHeading

\ShortArticleName{Fixed Point Algebras for Easy Quantum Groups}

\ArticleName{Fixed Point Algebras for Easy Quantum Groups}

\Author{Olivier GABRIEL~$^\dag$ and Moritz WEBER~$^\ddag$}

\AuthorNameForHeading{O.~Gabriel and M.~Weber}

\Address{$^\dag$~University of Copenhagen, Universitetsparken 5, 2100 K\o{}benhavn \O{}, Denmark}
\EmailD{\href{mailto:olivier.gabriel.geom@gmail.com}{olivier.gabriel.geom@gmail.com}}
\URLaddressD{\url{http://oliviergabriel.eu}}

\Address{$^\ddag$~Fachbereich Mathematik, Universit\"at des Saarlandes,\\
\hphantom{$^\ddag$}~Postfach 151150, 66041 Saabr\"ucken, Germany}
\EmailD{\href{mailto:weber@math.uni-sb.de}{weber@math.uni-sb.de}}

\ArticleDates{Received June 13, 2016, in f\/inal form September 26, 2016; Published online October 01, 2016}

\Abstract{Compact matrix quantum groups act naturally on Cuntz algebras. The f\/irst author isolated certain conditions under which the f\/ixed point algebras under this action are Kirchberg algebras. Hence they are completely determined by their $K$-groups. Building on prior work by the second author, we prove that free easy quantum groups satisfy these conditions and we compute the $K$-groups of their f\/ixed point algebras in a general form. We then turn to examples such as the quantum permutation group $S_n^+$, the free orthogonal quantum group $O_n^+$ and the quantum ref\/lection groups $H_n^{s+}$. Our f\/ixed point-algebra construction provides concrete examples of free actions of free orthogonal easy quantum groups, which are related to Hopf--Galois extensions.}


\Classification{46L80; 19K99; 81R50}

\begin{flushright}
\textit{In memory of the late Professor John E.~Roberts.}
\end{flushright}

\section{Introduction}

This article was initiated from the meeting of the two authors and their respective interests in easy quantum groups and the f\/ixed point algebra construction. Let us start by reminders on the setting of the present article.

\emph{Compact quantum groups} (CQGs) were def\/ined by Woronowicz and further studied in a series of papers \cite{CpctMPseudoGpWoronowicz,TwistedSU2Woronowicz,CpctQGpWoronowicz}. Following the paradigm of noncommutative geometry, the general idea is to describe all properties of a compact group $G$ in terms of its algebra $C(G)$ of (continuous) functions, using in particular a \emph{comultiplication} $\Delta \colon C(G) \to C(G \times G) \simeq C(G) \otimes C(G)$ to realise the group law $\mu \colon G \times G \to G$. If we then consider (possibly noncommutative) $C^*$-algebras with such a comultiplication, we get CQGs as an extension of compact groups. Of course, additional assumptions are needed to make the above rigorous (see Section~\ref{SSec:CMQG} below). To be more precise, we will mainly deal with \emph{compact matrix quantum groups} (CMQGs).

Among CMQGs, there is a class of particular examples, called \emph{easy quantum groups}. Categories of partitions and easy quantum groups were f\/irst def\/ined by Banica and Speicher in~\cite{BS09} in the orthogonal case. Tarrago and the second author extended their approach to the unitary setting, see~\cite{UEQGpTarragoWeber}. To each easy quantum group is associated a category of partitions, which provides a way to ``visualise'' it. The basic idea of easy quantum groups is that they should form a~tractable sub-class of CMQGs, since they can be described and manipulated \textsl{via} their category of partitions, by a~Tannaka--Krein type argument~\cite{TwistedSUWoronowicz}. This line of argument is illustrated by the article~\cite{FreslonWeber13}, where Freslon and the second author provided a description of fusion rules for easy quantum groups based on their categories of partitions.

On another note, the \emph{Cuntz algebra} $\mathcal{O}_n$ has been def\/ined as a universal $C^*$-algebra by Cuntz in his paper~\cite{CuntzAlg} and has evolved over time into one of the most important examples of $C^*$-algebras, with applications to classif\/ication theory and physics. An example of an application is provided by Doplicher and Roberts's abstract, Tannaka--Krein like duality results in a series of articles (see, e.g.,~\cite{CuntzAlgDR,EndomDR}) for actions of (ordinary) compact groups on $C^*$-algebras. This discovery motivated a considerable interest (see for instance~\cite{HilbertCSystBaumgLledo,DualThmPinzariRoberts,BraidedCatPinzariRoberts}). A~basic step of Doplicher--Roberts's duality theory is to consider so-called ``canonical actions'' of compact groups on Cuntz algebras. A source of inspiration for further research in this direction is the article~\cite{SimpleCAlgPinzari}, where Pinzari introduces a~f\/ixed point algebra $\CAlg_{\lambda(G)}$ from the regular representation $\lambda$ acting on $\CAlg_{L^2(G)}$ and proves that given two compact groups $G, G'$, the f\/ixed point algebras $\CAlg_{\lambda(G)}$ and $\CAlg_{\lambda(G')}$ are isomorphic as $\Z$-algebras if and only if $C^*(G) \simeq C^*(G')$.

Motivated by the desire of generalising Doplicher--Roberts theory and following the articles \cite{QGpActCPZ, CpctQGpMarciniak,CoactionCuntzAlgPaolucci}, the f\/irst author considered an action of a CQG $\CQG $ on a Cuntz algebra and described its f\/ixed point algebra. More precisely, two conditions~\ref{Cond:Contra} and \ref{Cond:Entrelac} were introduced, which ensure that the f\/ixed point algebra is actually a Kirchberg algebra in the UCT class $\UCTclass$. Kirchberg--Phillips's classif\/ication theory (see \cite{ClassThmKirchberg,ClassThmKirchbergKP}) then proves that up to $*$-isomorphism, the f\/ixed point algebra is characterised by its $K$-theory. In~\cite{FixedPtGabriel}, examples of computations of the $K$-theory of the f\/ixed point algebra are given~-- they only depend on the fusion rules of~$\CQG$.

In the present article, we combine these two directions of research to describe the f\/ixed point algebras of actions of the free orthogonal quantum group~$O_n^+$, the quantum permutation group~$S^+_n$ and the quantum ref\/lection group~$H^{s+}_n$. An interesting feature of the present f\/ixed point algebra construction is that it provides a very concrete realisation of the intertwiner spaces def\/ining the easy quantum group (see Proposition~\ref{Prop:Intertwiners} below).

The main results of this paper are the reformulation and characterisation of the hypotheses~\ref{Cond:Contra} and~\ref{Cond:Entrelac} of the f\/ixed point algebra construction theory of \cite{FixedPtGabriel} in terms of partition categories (see Theorems~\ref{Prop:Reform} and~\ref{Thm:CharactCP} below), together with the identif\/ications of $K$-theory for the f\/ixed point algebras associated to the natural representations of $O_n^+$, $S^+_n$ and $H^{s+}_n$ (see Theorems~\ref{Thm:A} and~\ref{Thm:Kth} below). The f\/ixed point algebras for~$O_n^+$ and~$S^+_n$ are isomorphic while the one for $H^{s+}_n$ is very dif\/ferent. Thus, in some sense, the actions of $S_n^+$ and $O_n^+$ are somehow ``similar'' while $H_n^{s+}$ acts very dif\/ferently. Moreover, since these f\/ixed point algebras depend only on the fusion rules of the CQGs at hand, this phenomenon manifests concretely that the fusion rules of~$H^{s+}_n$ are very dif\/ferent from those of~$O^+_n$ and~$S^+_n$.

This article is organised as follows: in Section~\ref{Sec:Reminders}, we start by a review of the notions of CQGs and CMQGs, before presenting the notions of categories of partitions and easy quantum groups and discussing Cuntz algebras and actions of CMQGs on these. Section~\ref{Sec:ActionsCuntzAlg} is devoted to the f\/ixed point algebra construction properly speaking, while Sections~\ref{Sec:Examples} and~\ref{Sec:QRefGp} are detailed studies of examples, namely the free orthogonal quantum group $O^+_n$ and the quantum permutation group~$S^+_n$ on the one hand, and the quantum ref\/lection group $H^{s+}_n$ on the other hand.

\section{Reminders and review}\label{Sec:Reminders}

\subsection{Compact matrix quantum groups}\label{SSec:CMQG}

In this article, we consider only \emph{minimal} tensor products of $C^*$-algebras. We will deal with \emph{compact quantum groups} (CQGs) which we denote by $\CQG$. They are def\/ined by a separable unital $C^*$-algebra $C(\CQG)$ together with a unital $*$-algebra homomorphism $\Delta\colon C(\CQG) \to C(\CQG) \otimes C(\CQG)$ which satisf\/ies coassociativity and cancellation properties -- for more details on these objects and their representations, see \cite{CpctMPseudoGpWoronowicz,CpctQGpWoronowicz}.

These compact quantum groups admit (unitary) representations or, equivalently \emph{actions} on Hilbert spaces and $C^*$-algebras. A \emph{unitary representation} $\brep$ of $\CQG $ on a Hilbert space $\Hh$ of f\/inite dimension $d$ is a $C(\CQG )$-valued $d \times d$ matrix $U = (\brep_{ij}) \in M_d(C(\CQG ))$ which is unitary and satisf\/ies the coassociativity property
\begin{gather*}
\Delta(\brep_{ij}) = \sum_{k} \brep_{ik} \otimes \brep_{kj}.
\end{gather*}
In particular, for any $\CQG $, we have a trivial representation denoted by $\TrivRep$, def\/ined by the $(1 \times 1)$-matrix $1 \in C(\CQG )$.

For two f\/ixed representations $u \in M_{d_1}(C(\CQG ))$ and $v \in M_{d_2}(C(\CQG ))$ acting on the Hilbert spaces $\Hh_1$ and $\Hh_2$, respectively, a linear map $T \colon \Hh_1 \to \Hh_2$ is an \emph{intertwiner} (see \cite{CpctMPseudoGpWoronowicz}) if
\begin{gather*} Tu=vT. \end{gather*}
We denote by $\Hom(u, v)$ the set of interwiners between $u$ and $v$. If $\Hom(u, v)$ includes an invertible map, then we say that $u$ and $v$ are \emph{equivalent}. Just like for ordinary compact groups, any representation of $\CQG $ is equivalent to a unitary one, therefore we will only consider unitary representations and we will refer to these as ``representations''. A~representation $\brep$ is called \emph{irreducible} if $\Hom(u, u) \simeq \C$. Representations of CQGs admit notions of direct sum~-- with the above notations, $u \oplus v$ is a representation on $\Hh_1 \oplus \Hh_2$~-- and tensor product~-- $u \otimes v$ is then a~representation on $\Hh_1 \otimes \Hh_2$.

 It is a well-known property of CQGs (see, e.g.,~\cite[Theorem~3.4]{CpctQGpWoronowicz}) that every unitary representation of a CQG is unitarily equivalent to a direct sum of irreducible unitary representations and any irreducible representation is f\/inite-dimensional. Given two representations~$u$ and~$v$, we use the notation $u \leqslant v$ to express that the representation $u$ is \emph{included} in~$v$ (i.e., there is an isometry in $\Hom(u, v)$).

Of particular interest for our investigations are \emph{compact matrix quantum groups} (CMQGs), a~particular class of CQGs. A CMQG is given by a so-called \emph{fundamental representation} $\brep$, whose coef\/f\/icients generate a dense subalgebra of $C(\CQG )$. One can def\/ine CMQGs in the following way:

\begin{Definition}\label{Def:CMQG}
A \emph{compact matrix quantum group} is def\/ined by a unital $C^*$-algebra $A$ generated by elements $u_{ij}$, $1\leq i,j\leq n$ such that $u=(u_{ij})$ and $u^t=(u_{ji})$ are invertible and the map
\begin{gather*}
\Delta\colon \ A\to A\otimes A,\qquad u_{ij}\mapsto \sum_k u_{ik}\otimes u_{kj}
\end{gather*}
is a $*$-homomorphism. A CMQG $(A,u)$ is a \emph{quantum subgroup} of $(B,v)$, if there is a surjective $*$-homomorphism from~$B$ to $A$ mapping $v_{ij}$ to~$u_{ij}$.
\end{Definition}

Of course, in the def\/inition of a CMQG $(A, u)$, we should think of $A$ as the functions on the ``quantum space'' $\CQG _A$ (and $(B,v)$ as associated to $\CQG _B$). Gelfand duality, which is contravariant, therefore explains why the ``inclusion $\CQG _A \inj \CQG _B$'' is represented by a \emph{surjective} morphism $B \to A$.

In this paper, we use the notations $\N_0 = \{ 0,1,2, \ldots \}$ and $\N = \{1, 2, \ldots \}$ (as opposed to \cite{FixedPtGabriel}, where $\N = \{ 0, 1, \ldots \}$).

\subsection{Categories of partitions and easy quantum groups}\label{Sect:Easy}

Colored partitions are key tools for the introduction of unitary easy quantum groups, as done in \cite{UEQGpTarragoWeber}. They generalize the orthogonal easy quantum groups of Banica and Speicher \cite{BS09}. Let $k,l\in\N_0$ and consider a f\/inite ordered set with $k+l$ elements each being colored either in white or in black. A \emph{partition} is a decomposition of this set into disjoint subsets, the \emph{blocks}. Let $P^{\twocol}(k,l)$ denote the set of all such partitions and put $P^{\twocol}:=\bigcup_{k,l\in \mathbb N_0} P^{\twocol}(k,l)$. We usually use a pictorial representation of a partition involving lines representing the block structure, and we assume $k$ of these points to be placed on an upper row and $l$ on a lower row, see \cite{TWcomb}. If these lines may be drawn in a way such that they do not cross, we call the partition \emph{noncrossing}. Let $NC^{\twocol}$ be the set of all noncrossing partitions. Some examples of partitions are the \emph{singleton partitions} $\singletonw\in P^{\twocol}(0,1)$ consisting of a single white lower point or $\singletonb\in P^{\twocol}(0,1)$, the \emph{pair partitions} $\paarpartwb$ and $\paarpartbw$ in $P(0,2)$ consisting of two lower points of dif\/ferent colors which are in the same block, the \emph{identity partitions} $\idpartww,\idpartbb\in P(1,1)$ consisting of one upper and one lower point both of the same color and both in the same block, or the partitions $b_s\in P^{\twocol}(0,s)$ consisting of $s$ lower white points in a single block.

We have the following operations on the set $P^{\twocol}$ of partitions. The \emph{tensor product} of $p\in P^{\twocol}(k, l)$ and $q\in P^{\twocol}(k', l')$ is $p\otimes q\in P^{\twocol}(k+k', l+l')$ obtained by placing $p$ and $q$ side by side. The \emph{composition} of $p\in P^{\twocol}(k, l)$ and $q\in P^{\twocol}(l, m)$ is $qp\in P^{\twocol}(k, m)$ obtained by pla\-cing~$p$ above~$q$. We may only perform it when the color pattern of the lower points of $p$ matches the upper color pattern of $q$. The \emph{involution} of $p\in P^{\twocol}(k, l)$ is $p^{*}\in P^{\twocol}(l, k)$ obtained by ref\/lec\-ting~$p$ at the horizontal axis. The \emph{rotation} of the left upper point of $p\in P^{\twocol}(k, l)$ to the lower row is a~partition in $P^{\twocol}(k-1, l+1)$. When rotating a point, its color is inverted but its membership to a block remains untouched. Likewise we have a rotation of the left lower or right upper/lower points. These operations (tensor product, composition, involution and rotation) are called the \emph{category operations}.

A collection $\CC$ of subsets $\CC(k, l)\subseteq P^{\twocol}(k, l)$ (for every $k, l\in\N_{0}$) is a \emph{category of partitions} if it is invariant under the category operations and if the identity partitions $\idpartww,\idpartbb\in P^{\twocol}(1, 1)$ and the pair partitions $\paarpartwb,\paarpartbw\in P^{\twocol}(0,2)$ are in $\CC$. Note that rotation may be deduced from the other category operations. We write $\CC=\langle p_1,\ldots, p_m\rangle$, if $\CC$ is the smallest category of partitions containing the partitions $p_1,\ldots, p_m$. We say that~$\CC$ is \emph{generated} by $p_1,\ldots,p_m$.

Let $n\in\N$. Given $p \in P^{\twocol}(k,l)$ and two multi-indices $(i(1), \dotsc, i(k))$, $(j(1), \dotsc, j(l))$ with entries in $\{1,\ldots,n\}$, we can label the diagram of $p$ with these numbers (the upper and the lower row both are labelled from left to right, respectively) and we put
\begin{gather*}\delta_p(i,j) =
 \begin{cases}
 1, & \text{if each block of $p$ connects only equal indices}, \\
 0, & \text{otherwise}.
 \end{cases}\end{gather*}
We f\/ix a basis $e_1,\ldots, e_n$ of $\C^n$ and def\/ine a map $T_p\colon (\C^n)^{\otimes k} \to (\C^n)^{\otimes l}$ associated to $p$ by
\begin{gather*}T_p(e_{i(1)} \otimes \dotsm \otimes e_{i(k)}) = \sum_{1 \leq j(1), \dotsc, j(l) \leq n} \delta_p(i, j) \cdot e_{j(1)} \otimes \dotsm \otimes e_{j(l)}. \end{gather*}
We use the convention that $(\C^n)^{\otimes 0}=\C$, i.e., for a partition $p\in P^{\twocol}(0,l)$ with no upper points, $T_p(1)$ is actually a vector in $(\C^n)^{\otimes l}$. Note that the colors of the points of $p$ do not play a role in the def\/inition of $T_p$.
The operations on the partitions match nicely with canonical operations of the linear maps $T_p$, namely we have $T_{p\otimes q}=T_p\otimes T_q$, $T_{qp}=n^{-b(p,q)}T_qT_p$ and $T_{p^*}=(T_p)^*$. Here, $b(p,q)$ denotes the number of removed blocks when composing $p$ and $q$.
The maps $T_p$ can be normalized in such a way that they become partial isometries, see~\cite{FreslonWeber13}.

From a category $\CC$ and the realisation $p\mapsto T_p$, we may construct a concrete monoidal $W^*$-category with a distinguished object $u$, and we thus may assign a CMQG $(A,u)$ to it. We call it the \emph{easy quantum group} associated to $\CC$. If $\CC\subset NC^{\twocol}$, then it is called a \emph{free easy quantum group}. Given a color string $r\in\{\circ,\bullet\}^k$, we may def\/ine $u^r$ as the tensor product of copies~$u$ and~$\bar u$. Here $\circ$ corresponds to $u$ while $\bullet$ corresponds to~$\bar u$. We then may say: a CMQG $S_n\subset \CQG\subset U_n^+$ is easy, if there is a category of partitions $\CC$ such that the intertwiner spaces of $\CQG$ are of the form
\begin{gather*}\{T\,|\, Tu^{r}=u^{s}T\}=\operatorname{span}\{T_p\,|\, p\in\CC(k,l) \textnormal{ with color strings $r$ (upper) and $s$ (lower)}\}.\end{gather*}

Easy quantum groups are a class of CMQGs with a quite intrinsic combinatorial structure. They are completely determined by their associated category of partitions and in many cases certain quantum algebraic properties of an easy quantum group may be traced back to certain combinatorial properties of its category of partitions. A simple criterion for verifying that a~given CMQG is easy is contained in the following lemma.

Let $n\in\N$ and let $A$ be a $C^*$-algebra generated by $n^2$ elements $u_{ij}$, $1\leq i,j\leq n$. Let $p\in P^{\twocol}(k,l)$ be a partition with upper color string $r\in\{\circ,\bullet\}^k$ and lower color string $s\in\{\circ,\bullet\}^l$. We say that the generators $u_{ij}$ \emph{fulfill the relations} $R(p)$, if for all $\beta(1),\ldots, \beta(l)\in\{1,\ldots,n\}$ and for all $i(1),\ldots,i(k)\in\{1,\ldots,n\}$, we have
\begin{gather*}\sum_{\alpha(1),\ldots,\alpha(k)=1}^n \delta_p(\alpha,\beta) u_{\alpha(1)i(1)}^{r_1}\cdots u_{\alpha(k)i(k)}^{r_k}=\sum_{\gamma(1),\ldots,\gamma(l)=1}^n \delta_p(i,\gamma) u_{\beta(1)\gamma(1)}^{s_1}\cdots u_{\beta(l)\gamma(l)}^{s_l}.
\end{gather*}
The left-hand side of the equation is $\delta_p(\varnothing,\beta)$ if $k=0$ and analogously $\delta_p(i,\varnothing)$ for the right-hand side if $l=0$. Furthermore, $u_{ij}^\circ:=u_{ij}$ and $u_{ij}^\bullet:=u_{ij}^*$.

\begin{lem}[{\cite[Corollary~3.12]{UEQGpTarragoWeber}}]
Let $p_1,\ldots,p_m\in P^{\twocol}$ be partitions and let $A$ be the universal $C^*$-algebra generated by elements $u_{ij}$, $1\leq i,j\leq n$ such that $u=(u_{ij})$ and $\bar u=(u_{ij}^*)$ are unitary $($i.e., $\sum_k u_{ik}^*u_{jk}=\sum_k u_{ki}^*u_{kj}=\sum_k u_{ik}u_{jk}^*=\sum_k u_{ki}u_{kj}^*=\delta_{ij})$ and $u_{ij}$ satisfy relations $R(p_l)$ for $l=1, \ldots , m$. Then $A$ is an easy quantum group with associated category $\CC=\langle p_1,\ldots, p_m\rangle$.
\end{lem}
\begin{proof}
The proof relies on the fact that the relations $R(p)$ are fulf\/illed if and only if $T_p$ intertwines $u^{r}$ and $u^{s}$. See \cite[Lemma~3.9, Corollary~3.12]{UEQGpTarragoWeber} for details.
\end{proof}

Note that for $p=\idpartwb$ the relation $R(p)$ ef\/fects that all $u_{ij}$ are selfadjoint. On the combinatorial level this amounts to having partitions whose points have no colors. In this case, we recover the orthogonal easy quantum groups of Banica and Speicher~\cite{BS09}.

\begin{ex}\label{Ex:WangEtc}\quad
\begin{itemize}\itemsep=0pt
\item[(a)] The \emph{free orthogonal quantum group} $O_n^+$ is given by the universal unital $C^*$-algebra $C(O_n^+)$ generated by selfadjoint elements $u_{ij}$, $1\leq i,j\leq n$ such that $u=(u_{ij})$ is an orthogonal matrix. It is easy with category $\CC=\langle \idpartwb\rangle=NC_2$, the set of all (noncolored) noncrossing pair partitions (each block consists of exactly two elements). We omit to write down the generating partitions $\paarpartwb$, $\paarpartbw$ and $\idpartww$, $\idpartbb$ since they are contained in every category, by def\/inition.
\item[(b)] The \emph{free unitary quantum group} $U_n^+$ is given by the universal unital $C^*$-algebra $C(U_n^+)$ generated by elements $u_{ij}$, $1\leq i,j\leq n$ such that $u=(u_{ij})$ and $\bar u=(u_{ij}^*)$ are unitary matrices. It is an easy quantum group with $\CC=\langle\varnothing\rangle$.
\item[(c)] The \emph{quantum permutation group} $S_n^+$ is given by the universal unital $C^*$-algebra $C(S_n^+)$ generated by projections $u_{ij}$, $1\leq i,j\leq n$ such that $\sum_ku_{ik}=\sum_ku_{kj}=1$. It is easy with category $\CC=\langle \idpartwb, \singletonw,\vierpartrotwwww\rangle$.
 \end{itemize}
    The quantum groups in (a), (b) and (c) were introduced by Wang~\cite{FreeProdCQGWang, QSymmGpFiniteSpWang}.
\begin{itemize}\itemsep=0pt
\item[(d)] For $s\in\N$, the \emph{quantum reflection group} $H_n^{s+}$ is given by the universal unital $C^*$-algeb\-ra~$C(H_n^{s+})$ generated by elements $u_{ij}$, $1\leq i,j\leq n$ such that $u=(u_{ij})$ and $\bar u=(u_{ij}^*)$ are unitaries, all $u_{ij}$ are partial isometries and we have $u_{ij}^s=u_{ij}u_{ij}^*=u_{ij}^*u_{ij}$. We have \mbox{$H_n^{1+}=S_n^+$} and \mbox{$H_n^{2+}=H_n^+$}, the latter one being Banica, Bichon and Collins's hyperoctahedral quantum group~\cite{HyperoctQG-BBC}. The quantum ref\/lection groups $H_n^{s+}$ were studied by Banica, Belinschi, Capitaine and Collins in~\cite{BBCC}. They are easy with $\CC=\langle b_s,\vierpartwwbb\rangle$, see \cite{UEQGpTarragoWeber}. Both~$H_n^+$ and~$H_n^{s+}$ can be traced back to Bichon's work on free wreath product \cite{FreeWreathBichon}.
 \end{itemize}
\end{ex}

\subsection{Cuntz algebra}\label{Sect:Cuntz}

The Cuntz algebra $\CAlg_n$ is the universal unital $C^*$-algebra generated by isometries $S_1,\ldots, S_n$ such that $\sum_i S_iS_i^*=1$. It has been def\/ined and studied by Cuntz in \cite{CuntzAlg}. He f\/irst proved that $\CAlg_n$ is a crossed product by an endomorphism and then introduced a criterion for (what we call today) purely inf\/inite and simple algebras. He used it to show that all $\CAlg_n$ are simple.

By an \emph{action} of a CQG on a $C^*$-algebra $A$ we mean a faithful unital $*$-homomorphism $\alpha\colon A\to A\otimes C(\CQG)$ such that $(\alpha\otimes\id)\circ\alpha=(\id\otimes\Delta)\circ\alpha$ holds and $(1\otimes C(\CQG))\alpha(A)$ is linearly dense in $A\otimes C(\CQG)$. In our case, we have the following action on the Cuntz algebra, as observed by Cuntz \cite{RegActionHopfCuntz} (see also \cite{ActCMQGKNW}).

\begin{prop}\label{Prop:Action}
Let $\CQG$ be a CMQG such that $\CQG\subset U_n^+$ $($i.e., the matrices $u$ and $\bar u$ are unitaries$)$. It acts on $\CAlg_n$ by
\begin{gather*} \alpha(S_i) = \sum_{j=1}^n S_j \otimes u_{ji}. \end{gather*}
\end{prop}
\begin{proof}
For the convenience of the reader, we give the proof since it is a short argument. By the universal property of $\CAlg_n$, the unital $*$-homomorphism $\alpha$ exists, and it is faithful since $\CAlg_n$ is simple. It is straightforward to check that $(\alpha\otimes\id)\circ\alpha=(\id\otimes\Delta)\circ\alpha$ is satisf\/ied, and the computation (using orthogonality of $u$)
\begin{gather*}\sum_k(1\otimes u_{ik}^*)\bigg(\sum_lS_l\otimes u_{lk}\bigg)=S_i\otimes 1\end{gather*}
shows that the linear span of $(1\otimes C(\CQG))\alpha(A)$ equals $A\otimes C(\CQG)$.
\end{proof}

It is a well-known fact that we may f\/ind copies of matrix algebras inside the Cuntz algebra. Indeed, the monomials $S_{j(1)}\cdots S_{j(k)}S_{i(k)}^*\cdots S_{i(1)}^*$ satisfy the relations of matrix units, thus we have
\begin{gather*}\operatorname{span}\{S_{j(1)}\cdots S_{j(k)}S_{i(k)}^*\cdots S_{i(1)}^*\,|\, 1\leq i(t),j(t)\leq d\}\cong B\big(\big(\C^n\big)^{\otimes k}\big).\end{gather*}
Thus, $S_{j(1)}\cdots S_{j(k)}S_{i(k)}^*\cdots S_{i(1)}^*$ corresponds exactly to the rank one operator mapping $e_{i(1)}\otimes \cdots \otimes e_{i(k)}$ to $e_{j(1)}\otimes \cdots \otimes e_{j(k)}$. More generally, we may identify linear maps
\begin{align*}
T\colon \ (\C^n)^{\otimes k}&\to (\C^n)^{\otimes l},\\
e_{i(1)}\otimes\cdots \otimes e_{i(k)}&\mapsto\sum_{j(1),\ldots, j(l)} a(i(1),\ldots, i(k),j(1),\ldots,j(l)) e_{j(1)}\otimes\cdots \otimes e_{j(l)}
\end{align*}
(where $a(i(1),\ldots, i(k),j(1),\ldots,j(l))\in\C$) with elements
\begin{gather*}\sum_{i(1),\ldots, i(k), j(1), \ldots, j(l)=1}a(i(1),\ldots, i(k),j(1),\ldots,j(l)) S_{j(1)}\cdots S_{j(l)} S_{i(k)}^*\cdots S_{i(1)}^*\end{gather*}
in $\CAlg_n$. Hence, we may view the maps $T_p:(\C^n)^{\otimes k}\to (\C^n)^{\otimes l}$ indexed by partitions $p\in P(k,l)$ as elements in $\CAlg_n$ via
\begin{gather*}T_p \leftrightarrow
\sum_{i(1),\ldots, i(k), j(1), \ldots, j(l)=1}^n \delta_p(i, j) S_{j(1)}\cdots S_{j(l)} S_{i(k)}^*\cdots S_{i(1)}^* .\end{gather*}

\section{Actions of easy quantum groups on the Cuntz algebra}\label{Sec:ActionsCuntzAlg}

\subsection{The f\/ixed point algebra}

As explained in Section~\ref{Sect:Cuntz}, any given CQMG acts naturally on a Cuntz algebra $\CAlg_n$ for a suitable choice of $n$. Our aim is to understand the f\/ixed point algebra in the case of easy quantum groups.

\begin{Definition}Let $\CQG$ be a CMQG with fundamental representation $u=(u_{ij})$ and let $\alpha$ be the action as in Proposition~\ref{Prop:Action}. The \emph{fixed point algebra} $\CAlg^\alpha$ is def\/ined as
\begin{gather*}
\CAlg^\alpha:= \{x\in \CAlg_n\,|\, \alpha(x)=x\otimes 1\}.
\end{gather*}
\end{Definition}

In the case of easy quantum groups, we may read the f\/ixed point algebras directly from the categories of partitions. The following statement appeared in \cite[Proposition~3.4]{CpctQGpMarciniak}, see also \cite[Lemma~2.5]{FixedPtGabriel}.

\begin{prop}\label{Prop:Intertwiners}\looseness=-1
Given $\CC$, let $\CQG$ be the associated easy QG. The intersections of the fixed point algebra $\CAlg^\alpha$ with the copies of $B((\C^n)^{\otimes k},(\C^n)^{\otimes l})$ in $\CAlg_n$ $($as described in Section~{\rm \ref{Sect:Cuntz})} are given by
\begin{gather*}
\CAlg^{\alpha}\cap B\big((\C^n)^{\otimes k},(\C^n)^{\otimes l}\big)=\operatorname{span}\{T_p\in \CAlg_n\,|\, p\in\CC(k,l) \textnormal{ all points are white}\}.
\end{gather*}
\end{prop}
\begin{proof}The general proof may be found in \cite[Proposition~3.4]{CpctQGpMarciniak}, but we give a direct argument here. Let $p\in\CC(k,l)$ be a partition with only white points. Hence, in $C(\CQG )$ the relations $R(p)$ hold. Moreover, we use the fact that $u$ is unitary. Now
\begin{gather*}
\alpha(T_p) = \sum_{i,i',j,j'} \delta_p(i, j) S_{j'(1)}\cdots S_{j'(l)}S_{i'(k)}^*\cdots S_{i'(1)}^* \otimes u_{j'(1)j(1)}\cdots u_{j'(l)j(l)}u_{i'(k)i(k)}^*\cdots u_{i'(1)i(1)}^*\\
\hphantom{\alpha(T_p)}{}
= \sum_{i',j'} S_{j'(1)}\cdots S_{j'(l)}S_{i'(k)}^*\cdots S_{i'(1)}^*\\
\hphantom{\alpha(T_p) =}{}
\otimes\bigg(\sum_i\bigg(\sum_{j}\delta_p(i, j)u_{j'(1)j(1)}\cdots u_{j'(l)j(l)}\bigg)u_{i'(k)i(k)}^*\ldots u_{i'(1)i(1)}^*\bigg)\\
\hphantom{\alpha(T_p)}{}
= \sum_{i',j'} S_{j'(1)}\cdots S_{j'(l)}S_{i'(k)}^*\cdots S_{i'(1)}^*\\
\hphantom{\alpha(T_p) =}{}
\otimes \bigg(\sum_i\bigg(\sum_{s}\delta_p(s, j')u_{s(1)i(1)}\cdots u_{s(k)i(k)}\bigg)u_{i'(k)i(k)}^*\cdots u_{i'(1)i(1)}^*\bigg)
\\
\hphantom{\alpha(T_p)}{}
= \sum_{i',j'} S_{j'(1)}\cdots S_{j'(l)}S_{i'(k)}^*\cdots S_{i'(1)}^*\\
\hphantom{\alpha(T_p) =}{}
\otimes \bigg(\sum_{s}\delta_p(s, j')\bigg(\sum_iu_{s(1)i(1)}\cdots u_{s(k)i(k)}u_{i'(k)i(k)}^*\cdots u_{i'(1)i(1)}^*\bigg)\bigg)\\
\hphantom{\alpha(T_p)}{}
= \sum_{i',j'} S_{j'(1)}\cdots S_{j'(l)}S_{i'(k)}^*\cdots S_{i'(1)}^* \otimes \bigg(\sum_{s}\delta_p(s, j')\delta_{si'}\bigg)\\
\hphantom{\alpha(T_p)}{}
= \sum_{j'} \delta_p(i', j') S_{j'(1)}\cdots S_{j'(l)}S_{i'(k)}^*\cdots S_{i'(1)}^* \otimes 1=T_p\otimes 1.
\end{gather*}

Conversely, let $T\in \CAlg^{\alpha}\cap B((\C^n)^{\otimes k},(\C^n)^{\otimes l})$. Reversing the above computation, we see that~$T$ intertwines $u^{\otimes k}$ and $u^{\otimes l}$. But $\Hom(u^{\otimes k},u^{\otimes l})$ consists exactly of the linear span of all $T_p\in \CAlg_n$ such that $p\in\CC(k,l)$ has only white points.
\end{proof}

\begin{prop}
The algebraic result of Proposition~{\rm \ref{Prop:Intertwiners}} extends into a topological one, namely,
\begin{gather*}
\CAlg^{\alpha} = \overline{\operatorname{span}\{T_p\in \CAlg_n\,|\, p\in\CC(k,l) \; \text{all points are white}, \; \forall \, k, l \in \N_0\}}.
\end{gather*}
\end{prop}

\begin{proof}
Lemma~2.7 of \cite[p.~1017]{FixedPtGabriel} (see also \cite[Lemma~6]{ActCMQGKNW}) ensures that ``algebraic elements'' (i.e., those obtained by (f\/inite) polynomial combination of $S_j$) of $\CAlg^\alpha$ are dense in $\CAlg^\alpha$. Proposition~\ref{Prop:Intertwiners} above ensures that any such algebraic element is in the linear span of $T_p$, so the conclusion follows.
\end{proof}

\subsection{Obtaining Kirchberg algebras}

In \cite{FixedPtGabriel}, the f\/irst author isolated two conditions which turn $\CAlg^\alpha$ into a Kirchberg algebra. This class of algebras plays a central role in Kirchberg--Phillips classif\/ication theory, which proves that Kirchberg algebras are completely classif\/ied by their $K$-groups -- see \cite{ClassThmKirchberg,ClassThmKirchbergKP} for the original papers, \cite{RordamStormer} for an overview of classif\/ication for nuclear simple $C^*$-algebras and \cite{QuasidiagTWW} for the latest developments in this area.

We denote by $\CatInd$ the set of (classes of) irreducible representations appearing in the iterated tensor products $\brep^{\otimes l}$ for $l \in \N_0$.
\begin{enumerate}[label=(C\arabic*)]\itemsep=0pt
\item\label{Cond:Contra}
For any $v \in \CatInd $, we can f\/ind $v' \in \CatInd$ such that the representation $v \otimes v'$ possesses a nonzero invariant vector.
\item\label{Cond:Entrelac}
There are integers $N, k_0 \in \N_0$ such that $\brep^{\otimes N}$ is contained in $\brep^{\otimes (N+k_0)}$ and for all inte\-gers~$t$,~$l$ with $0 < t < k_0$, $\Hom(\brep^{\otimes l},\brep^{\otimes(l+t)}) = 0$.
\end{enumerate}

Condition \ref{Cond:Contra} is actually satisf\/ied for all f\/inite-dimensional semisimple Hopf algebras over~$\C$, see \cite[Remark~7.5]{FixedPtGabriel} and \cite[Theorem~of Section~4.2]{HigherFrobSchurKSZ}. The following result has been proven by the f\/irst author, see \cite[Lemma~2.10, Corollary~4.7, Lemma~6.4]{FixedPtGabriel}. It holds true in a much more general setting but we restrict it to CMQG.

\begin{prop}\label{Prop:PrevResults}
Let $\CQG$ be a CMQG and let $\alpha$ be its action on $\CAlg_n$ as described in Proposition~{\rm \ref{Prop:Action}}. If the conditions~{\rm \ref{Cond:Contra}} and~{\rm \ref{Cond:Entrelac}} are satisfied, then the fixed point algebra $\CAlg^\alpha$ is a~Kirchberg algebra, i.e., it is purely infinite, simple, separable, unital and nuclear $($satisfying the UCT$)$.
\end{prop}

We are now going to study which easy quantum groups satisfy conditions~\ref{Cond:Contra} and~\ref{Cond:Entrelac}. In order to do so, let us rephrase these conditions in the language of partitions. The representation theory of easy quantum groups -- i.e., the set~$\CatInd$ -- is completely understood and can be given in terms of partitions, see~\cite{FreslonWeber13}. We review some facts from~\cite{FreslonWeber13}.

We say that a partition $p\in P^{\twocol}(k,k)$ is \emph{projective}, if $p=p^*=p^2$. It follows that $T_p$ is a~projection up to normalization. Moreover, the upper points of $p$ are colored exactly like the lower points of~$p$. If $q\in P^{\twocol}(k,k)$ is another projective partition, we write $q\prec p$ if $pq=qp=q$ and $p\neq q$. In this case, $T_q$ is a subprojection of~$T_p$. Given a category $\CC$ of partitions, we write $\Proj_{\CC}(k)$ for the set of all projective partitions in $\CC(k,k)$. For $p\in \Proj_{\CC}(k)$ with upper (and equivalently lower) color string~$s$, we put
\begin{gather*}
R_{p} := \bigvee_{q\in \Proj_{\CC}(k),\; q \prec p} T_{q}, \qquad \text{and} \qquad
P_{p} := T_{p} - R_{p} \in \Hom(u^{s}, u^{s}).
\end{gather*}
We denote by $u_{p}$ the subrepresentation $(\id\otimes P_{p})(u^s)$ of $u^s$. Two such subrepresenta\-tions~$u_p$ and~$u_q$ are unitarily equivalent if and only if there is a partition~$r\in\CC$ such that $p=r^*r$ and $q=rr^*$.

\begin{prop}[{\cite[Theorem~5.5]{FreslonWeber13}}]\label{Prop:FreslonWeber} If $\CC\subset NC^{\twocol}$, then $u_p$ is irreducible. Furthermore, any irreducible representation of the associated~$\CQG $ is unitarily equivalent to some~$u_p$ for some $p\in \Proj_{\CC}$, thereby inducing a one-to-one correspondence.
\end{prop}

Since only tensor powers of $u$ play a role for conditions \ref{Cond:Contra} and \ref{Cond:Entrelac}, let us denote by $\Proj_{\CC}^{\white}$ and $\CC^{\white}$ the restrictions of $\Proj_{\CC}$ and $\CC$ to those partitions consisting only of white points.

\begin{thm}\label{Prop:Reform}
Let $\CQG$ be a free easy quantum group, namely the associated category of parti\-tions~$\CC$ contains only noncrossing partitions. The following two conditions imply the conditions~{\rm \ref{Cond:Contra}} and~{\rm \ref{Cond:Entrelac}}.
\begin{enumerate}[label={\rm (C$_P$\arabic*)}]\itemsep=0pt
\item\label{Cond:ContraPart}\looseness=-1
For any projective partition $p\in \Proj^{\white}_{\CC}(a)$, we can find a projective partition $q\!\in\!\Proj^{\white}_{\CC}(b)$ and a partition $r\!\in\!\CC^{\white}(0,a{+}b)$ with no upper points such that $(P_p\otimes P_q)T_r\!\neq\! 0$.
\item\label{Cond:EntrelacPart}
There are integers $k_0, N \in \N_0$ and a partition $r\in \CC^{\white}(N+k_0,N)$ such that $rr^*=\idpartww^{\otimes N}$. Moreover, for all $t\in\N_0$ with $0 < t < k_0$ and for all $l\in\N_0$ we have $\CC^{\white}(l,l+t)=\varnothing$.
\end{enumerate}
\end{thm}
\begin{proof} Due to Proposition~\ref{Prop:FreslonWeber}, the set $\CatInd$ is in bijection with equivalence classes of projective partitions in $\CC$. Let us f\/irst prove that~\ref{Cond:ContraPart} implies~\ref{Cond:Contra}. Let $p\in \Proj_{\CC}$ be a projective partition with only white points and let $u_p\in\CatInd$ be the associated irreducible representation. Let $q\in\Proj_{\CC}$ and $r\in\CC(0,a+b)$ be partitions according to~\ref{Cond:ContraPart}. Then $T_r(1)$ is a vector in $(\C^n)^{\otimes a+b}$ as described in Section~\ref{Sect:Easy} and hence $(P_p\otimes P_q)T_r(1)$ is a non-zero vector. Now, since~$P_p$,~$P_q$ and $T_r$ are intertwiners for tensor powers of $u$, we have:
\begin{gather*}
(u_p\otimes u_q)(P_p\otimes P_q)T_r
=(\id\otimes(P_p\otimes P_q))\big(u^{\otimes a+b}\big)(P_p\otimes P_q)T_r\\
\hphantom{(u_p\otimes u_q)(P_p\otimes P_q)T_r}{}
 =(\id\otimes(P_p\otimes P_q))\big(u^{\otimes a+b}\big)T_r =(P_p\otimes P_q)T_r.
\end{gather*}
This proves \ref{Cond:Contra}. As for deducing \ref{Cond:Entrelac} from \ref{Cond:EntrelacPart}, observe that $u^{\otimes N}T_r=T_ru^{\otimes N+k_0}$ implies $u^{\otimes N}T_rT_r^*=T_ru^{\otimes N+k_0}T_r^*$. Since $T_rT_r^*=1$ up to normalization, this proves that $u^{\otimes N}$ is a subrepresentation of $u^{\otimes N+k_0}$. Moreover, $\Hom(\brep^{\otimes l},\brep^{\otimes(l+t)})$ is spanned by all $T_p$ with $p\in\CC^{\white}(l,l+t)$. We have proved~\ref{Cond:Entrelac}.
\end{proof}

The advantage of dealing with free easy quantum groups is that they are all known: the class of categories $\CC\subset NC^{\twocol}$ -- and hence the class of free easy quantum groups -- is completely classif\/ied. The complete list may be found in \cite[Section~7]{TWcomb}.

Throughout the classif\/ication process, a parameter $k(\CC)\in \N_0$ is assigned to any category of partitions $\CC$, see \cite[Def\/inition~2.5]{TWcomb}. It is given by the following. For a partition $p\in P^{\twocol}$ denote by $c_\circ(p)$ the sum of the number of white points on the lower line of $p$ and the number of black points on its upper line. Likewise put $c_\bullet(p)$ to be the number of lower black points plus the number of upper white points. Let $c(p):=c_\circ(p)-c_\bullet(p)$. The number $c(p)$ is designed in such a way that it yields the dif\/ference of the number of white and black points, if $p$ has no upper points; moreover $c(p)$ is invariant under rotation. Now, let $k(\CC)$ be the minimum of all numbers $c(p)>0$ for $p\in\CC$ if such a number exists, and $k(\CC):=0$ otherwise. One can show that for every partition $p\in\CC$ the number $c(p)$ is a multiple (from the integers) of $k(\CC)$. One of the main results in \cite{UEQGpTarragoWeber} is then to detect the cyclic group of order $k(\CC)$ as a building block in the easy quantum group associated to~$\CC$. The parameter $k(\CC)$ also appears in connection with conditions~\ref{Cond:ContraPart} and~\ref{Cond:EntrelacPart}.

\begin{thm}\label{Thm:CharactCP}
Let $\CQG$ be an easy quantum group with $\CC\subset NC^{\twocol}$.
\begin{itemize}\itemsep=0pt
\item[$(a)$] If $k(\CC)=0$, then neither condition~{\rm \ref{Cond:ContraPart}} nor~{\rm \ref{Cond:EntrelacPart}} are satisfied.
\item[$(b)$] If $k(\CC)\neq 0$ and $\paarpartww\otimes\paarpartbb\in\CC$, then condition~{\rm \ref{Cond:ContraPart}} holds.
\item[$(c)$] If $k(\CC)\neq 0$ and $\vierpartwwbb\in\CC$, then condition~{\rm \ref{Cond:ContraPart}} holds.
\item[$(d)$] If $k(\CC)\neq 0$, then condition~{\rm \ref{Cond:EntrelacPart}} holds.
\end{itemize}
\end{thm}
\begin{proof}
(a) Observe that if a partition $p$ is in $\CC^{\white}(l,l+t)$, then $c(p)=t$. Thus, if $k(\CC)=0$, all sets $\CC^{\white}(l,l+t)$ are empty for all $l\in\N_0$ and all $t\neq 0$, using \cite[Proposition~2.7]{TWcomb}. Therefore neither condition~\ref{Cond:ContraPart} nor condition~\ref{Cond:EntrelacPart} hold.

(b) Let $p\in\Proj^{\white}_{\CC}(a)$ be a projective partition. We basically want to show that the contragredient representation of $u_p$ does the job for choosing $q$, but the colorization of the points turns this into a nontrivial problem (see also the next Examples~\ref{Ex:CP} and the remarks in~\cite[p.~1019]{FixedPtGabriel} on condition~\ref{Cond:Contra}). Let $q_1$ be the partition obtained by ref\/lecting $p$ about the vertical axis (without inverting the colors). It is in~$\CC$, since the verticolor ref\/lected partition $\tilde p$ is in $\CC$ \cite[Lemma~1.1(a)]{TWcomb}, and $p\otimes p\otimes \tilde p$ is in~$\CC$; using color permutation \cite[Lemmas~1.3(a) and~1.1(b)]{TWcomb}, we infer $q_1\in\CC$. Now, let $b\in\N_0$ be such that~$2(a+b)$ is a~multiple of~$k(\CC)$. Let $q_0=\paarpartww^{\nest(b)}\otimes \baarpartww^{\nest(b)} $ be the partition obtained by nesting $b$ copies of the pair partition~$\paarpartww$ into each other, both on the upper and the lower line respectively, see \cite[Lemma~2.4]{UEQGpTarragoWeber}. It is in~$\CC$ since we may apply \cite[Lemmas~1.3(a) and~1.1(b)]{TWcomb} on the following partition which is in~$\CC$:
\begin{gather*}\paarpartwb^{\nest(b)}\otimes \baarpartwb^{\nest(b)}\otimes (\paarpartww\otimes\baarpartww)^{\otimes b}.\end{gather*}
We then put $q:=q_0\otimes q_1\in\Proj^{\white}_{\CC}(a+2b)$ and $r:=\paarpartww^{\nest(a+b)}$. We have $r\in\CC^{\white}(0,2(a+b))$, see for instance Step 3 in the proof of \cite[Theorem~4.13]{UEQGpTarragoWeber}.

If now $s\in\Proj^{\white}_{\CC}(a)$ is a projective partition with $s\prec p$, then $(s\otimes q)r\neq (p\otimes q)r$ by \cite[Lemma~2.23]{FreslonWeber13}. Hence $T_{(s\otimes q)r}$ is linearly independent from $T_{(p\otimes q)r}$ by \cite[Lemma~4.16]{FreslonWeber13}. Likewise $T_{(p\otimes t)r}$ is linearly independent from $T_{(p\otimes q)r}$ for $t\prec q$, $t\in\Proj^{\white}_{\CC}(a+2b)$. This proves that $(T_p\otimes T_q)T_r$ is linearly independent from $(R_p\otimes T_q)T_r$, $(T_p\otimes R_q)T_r$ and $(R_p\otimes R_q)T_r$. Thus, $(P_p\otimes P_q)T_r\neq 0$.

(c) Let $p\in\Proj^{\white}_{\CC}(a)$ be a projective partition. Using the through-block decomposition of~$p$, we may bring it into the following form (see Proposition~2.9 and the remarks after Proposition~2.12 in~\cite{FreslonWeber13}):
\begin{gather*}p=p_u^*p_u,\qquad p_u=s_0\otimes t_1\otimes s_1\otimes\cdots \otimes t_l\otimes s_l.\end{gather*}
Here, $s_i$ are partitions with no upper points while each $t_i$ has exactly one upper point. Let $\alpha_i$ be the sum of the number of points of $s_i$ and $t_i$, with $\alpha_1$ being the sum of the number of points of~$s_0$,~$t_1$ and $s_1$. Let $\beta_i$ be numbers such that $\alpha_i+\beta_i$ is a multiple of $k(\CC)$, for all $i$. Let $q$ be the partition
\begin{gather*}q:=q_1\otimes\cdots\otimes q_l,\end{gather*}
where each $q_i$ is the partition consisting of a single block on $\beta_i$ upper white points and $\beta_i$ lower white points. Since $\vierpartwwbb\in\CC$, we have $q\in\CC$, applying \cite[Lemma~1.3(c)]{TWcomb} on $\idpartww^{\otimes \beta_i}$. Let $r\in\CC^{\white}(0,\sum_i(\alpha_i+\beta_i))$ be the partition obtained from nesting $l$ blocks $b_i$ into each other, each of size $\alpha_i+\beta_i$, such that block $b_{i+1}$ has $\alpha_i$ legs of $b_i$ to its left and $\beta_i$ legs of~$b_i$ to its right. We then conclude $(P_p\otimes P_q)T_r\neq 0$ similarly to~(b).

(d) Put $N:=1$ and $k_0:=k(\CC) > 0$. We f\/ind a partition $p \in \CC$ such that $c(p) = k_0$, by def\/inition of $k(\CC)$. Using rotation and \cite[Lemma~2.6(d)]{TWcomb}, we may assume that $p$ consists only of lower points, hence $p \in P^{\twocol}(0, m)$ for some $m \geqslant k_0$. Then $m = c_\circ(p)+c_\bullet(p)$ and $k_0 = c(p) = c_\circ(p)-c_\bullet(p)$, thus $p$ has $\frac{m+k_0}{2}$ white points and $\frac{m-k_0}{2}$ black points.

Using \cite[Lemma~1.1(b)]{TWcomb}, we may erase $\frac{m-k_0}{2}$ pairs of a~white and a~black point and we obtain a partition $p_0 \in \CC^{\white}(0, k_0)$ on $k_0$ white points. Put $r:=\idpartww\otimes p_0^*\in \CC^{\white}(1+k_0,1)$. It satisf\/ies $rr^*=\idpartww^{\otimes N}$ for $N = 1$.

Moreover, $\CC^{\white}(l,l+t)=\varnothing$ for $0<t<k(\CC)$, since if we had $p \in \CC^{\white}(l,l+t)$, then $c(p) = t$, but $c(p)$ is an integer multiple of $k(\CC)$ by \cite[Proposition~2.7]{TWcomb}.
\end{proof}

\begin{ex}\label{Ex:CP}\quad
\begin{itemize}\itemsep=0pt
\item[(a)] If $\CQG$ is free orthogonal easy ($u=\bar u$) with category $\CC\subset NC^{\twocol}$, then $k(\CC)=1$ if $\singletonw\in\CC$ and $k(\CC)=2$ otherwise, see \cite[Section~7]{TWcomb}. Moreover, as $\idpartwb$ is in $\CC$, we also have $\paarpartww\otimes\paarpartbb\in\CC$. Hence, conditions \ref{Cond:ContraPart} and \ref{Cond:EntrelacPart} hold for $S_n^+$, $O_n^+$ and the other f\/ive free orthogonal easy quantum groups.
\item[(b)] For $U_n^+$, we have $k(\CC)=0$, thus conditions~\ref{Cond:ContraPart} and~\ref{Cond:EntrelacPart} are violated.
\item[(c)]
The quantum ref\/lection groups $H_n^{s+}$ of Example~\ref{Ex:WangEtc} have the parameter $k(\CC)=s$ and $\vierpartwwbb\in\CC$, thus conditions~\ref{Cond:ContraPart} and~\ref{Cond:EntrelacPart} hold.
\end{itemize}
\end{ex}

Theorem~\ref{Thm:CharactCP} is about as far as we can go for free easy quantum groups in general: indeed, if Theorem~\ref{Thm:CharactCP} together with Proposition~\ref{Prop:PrevResults} provide a way to prove that some f\/ixed point algebras are Kirchberg algebras, the exact identif\/ication of the f\/ixed point algebra requires a~computation of $K$-theory that can only be performed for def\/inite fusion rules. For this reason, we focus on examples in the rest of the paper.

\subsection{Free actions}

In this subsection, we investigate the relation between our f\/ixed point construction and free actions. Free actions appear in articles such as \cite{FreeActBCH,NCjoinDHW, FreeActCAlgCY}. However, there are not that many concrete examples available, and for this reason we prove that our construction generates new examples.

We remind the reader of the following Def\/inition 2.4 of \cite{PpalActionEllwood} (see also~\cite{FreeActCAlgCY}):
\begin{Definition}Given a CQG $\CQG $, an action $\actionG \colon A \to A \otimes C(\CQG )$ on a $C^*$-algebra $A$ is called \emph{free} if $ (A \otimes 1) \actionG(A)$ is dense in $A \otimes C(\CQG )$.
\end{Definition}

In order to discuss more easily the above def\/inition, we follow the terminology and notations of \cite{FreeActBCH,NCjoinDHW,PPpalComodHKMZ} and introduce the \emph{canonical map} $\can \colon A \otimes A \to A \otimes C(\CQG )$ given by $\can(a \otimes a') := (a \otimes 1) \actionG(a')$. The condition above is therefore that the map $\can$ has a dense image.

\begin{Remark}The ($C^*$-)freeness property def\/ined above has strong ties with the notion of \emph{Hopf--Galois extension} (f\/irst def\/ined in~\cite{HGExtensionKT}~-- see for instance~\cite{SurveyHopfGaloisMontgomery}). Indeed, Woronowicz proved~\cite{CpctQGpWoronowicz} that any CQG contains a canonical dense Hopf $*$-algebra $\mathcal{O}(\CQG )$. Using the action of~$\CQG$ on~$A$, we can in turn def\/ine the \emph{Peter--Weyl subalgebra} $\mathcal{P}_\CQG(A)$ of $A$ (see, e.g.,~\cite{FreeActBCH})~-- which is not a $C^*$-algebra but just a $*$-algebra. The preprint \cite[p.~3]{FreeActBCH} proves in its Theorem~0.4 that the action $\actionG$ is ($C^*$-)free if and only if $\mathcal{P}_\CQG(A)$ is a Hopf--Galois extension over its f\/ixed point algebra~-- compare the \emph{Peter--Weyl--Galois condition} of \cite[Def\/inition~0.2, p.~3]{FreeActBCH} with Def\/inition~2.2 of Hopf--Galois extensions in \cite[p.~372]{SurveyHopfGaloisMontgomery}.
\end{Remark}

Going back to our initial motivations concerning free actions, we prove:
\begin{prop}\label{PropA}
Let $\CQG $ be a compact matrix quantum group and $\brep $ its fundamental representation. If $\CQG $ satisfies condition~{\rm \ref{Cond:Contra}} then the action $\actionG$ induced from Proposition~{\rm \ref{Prop:Action}} is free.
\end{prop}

\begin{Remark}At this point, it may seem plausible that condition~\ref{Cond:Contra} is equivalent to freeness of the action of $\CQG $ on $\CAlg_n$. However, as we have seen in \cite[Notation~3.1]{FixedPtGabriel}, \ref{Cond:Contra} is not satisf\/ied for the natural representation of $ \CQG ={\rm U}(1)$ (or multiple thereof). Using the same kind of argument as in the proof below, it appears that the action of~${\rm U}(1)$ on $\CAlg_2$ is free.

Of course, this action is not strictly induced from ${\rm U}(1)$ by Proposition~\ref{Prop:Action}. But we can build on this example by considering ${\rm SU}(2) \times {\rm U}(1)$. This is a CMQG which does not satisfy condition~\ref{Cond:Contra}, but which still acts freely on~$\CAlg_2$.
\end{Remark}

\begin{proof}[Proof of Proposition~\ref{PropA}] For our proof, we will consider the set
\begin{gather*}
\Ss := \left\{ s \in C(\CQG ) \,|\, \exists\, k_m \in \N_0, \; v_m, w_m \in \Hh^{\otimes k_m},\; \can\left( \sum_m^\text{f\/inite} v_m^* \otimes w_m\right) = 1 \otimes s \right\}.
\end{gather*}
 It is clear from the def\/inition of $\actionG$ that for any $i$, $j$,
\begin{gather*}
\can( S_j^* \otimes S_i ) = (S_j^* \otimes 1) \left( \sum_{k=1}^n S_k \otimes \brep _{ki} \right) = 1 \otimes \brep_{ji}.
\end{gather*}
This means that $\Ss$ contains all $\brep_{ji}$. Now, if $v_{i,m}$, $w_{i,m}$ are elements in $\Hh^{\otimes k_i}$ for $i=1,2$ and such that $\can\big(\sum_mv_{i,m}^* \otimes w_{i,m}\big) = 1 \otimes s_i$, then
\begin{gather*}
\can\bigg(\sum_{m,l}v_{2,m}^* v_{1,l}^* \otimes w_{1,l} w_{2,m}\bigg)
= \bigg(\sum_{m,l}v_{2,m}^* v_{1,l}^* \otimes 1\bigg) \actionG(w_{1,l} w_{2,m})\\
\hphantom{\can\bigg(\sum_{m,l}v_{2,m}^* v_{1,l}^* \otimes w_{1,l} w_{2,m}\bigg)}{}
= \bigg(\sum_mv_{2,m}^* \otimes 1\bigg) \bigg(\sum_lv_{1,l}^* \otimes 1\bigg) \actionG(w_{1,l}) \actionG(w_{2,m})\\
\hphantom{\can\bigg(\sum_{m,l}v_{2,m}^* v_{1,l}^* \otimes w_{1,l} w_{2,m}\bigg)}{}
 = \bigg(\sum_mv_{2,m}^* \otimes 1\bigg)(1 \otimes s_1) \actionG(w_{2,m}) \\
\hphantom{\can\bigg(\sum_{m,l}v_{2,m}^* v_{1,l}^* \otimes w_{1,l} w_{2,m}\bigg)}{}
= (1 \otimes s_1) \bigg(\sum_mv_{2,m}^* \otimes 1\bigg) \actionG(w_{2,m})\\
\hphantom{\can\bigg(\sum_{m,l}v_{2,m}^* v_{1,l}^* \otimes w_{1,l} w_{2,m}\bigg)}{}
= (1 \otimes s_1)(1 \otimes s_2) = 1 \otimes s_1 s_2.
\end{gather*}

This in turn proves that $\Ss$ is stable under multiplication. It follows that $\Ss$ is an algebra, which contains all polynomials in $\brep_{ij}$.

If the generators of $\CQG $ are selfadjoint, the polynomials in $\brep_{ij}$ coincide with the $*$-polynomials in $\brep_{ij}$ and they are all contained in $\Ss$, thus $\Ss$ is dense in $C(\CQG )$.

\looseness=-1
In general, we can always decompose $\brep$ into irreducible representations, which all correspond to unitary matrices with coef\/f\/icients in $C(\CQG )$. Up to a conjugation by a \emph{scalar-valued} unitary matrix, the algebra generated by the entries of $\brep$ and their adjoints is the same as the algebra generated by the irreducible components and their adjoints. So without loss of generality, we assume that $\brep$ is an irreducible representation. It is well known (see, e.g., \cite[Def\/inition~1.3.8, p.~11]{CQGRC-NT}) that the contragredient of $(\brep_{ij})$ is $(\brep_{ij}^*)$ in a suitable basis. $\brep$ is irreducible, so condition \ref{Cond:Contra} implies that up to equivalence, we can recover the coef\/f\/icients $(\brep_{ij}^*)$ through a subrepresentation of a~high enough power $\Hh^{\otimes k}$. It follows that $\Ss$ contains the $*$-polynomials in $\brep_{ij}$ and thus is dense in~$C(\CQG )$.

Given any $T \in \CAlg_n$, for any $\can\big(\sum v_j^* \otimes w_j\big) = 1 \otimes s$ we have
\begin{gather*} \can\bigg(\sum T v_j^* \otimes w_j\bigg) = (T \otimes 1) \can\bigg(\sum v_j^* \otimes w_j\bigg) = T \otimes s.\end{gather*}
It follows that the image of $\can$ is dense in $\CAlg_n \otimes C(\CQG )$ and thus the action $\actionG$ is free.
\end{proof}

\begin{Remark}On the one hand, it follows from the proof above that assumption~\ref{Cond:Contra} is not needed for Proposition~\ref{PropA} in the orthogonal case (i.e., if all the entries of the fundamental representation are selfadjoint). On the other hand, the argument above also shows that in this orthogonal case condition~\ref{Cond:Contra} is automatically satisf\/ied.
\end{Remark}

\begin{thm}
The actions of $S_n^+$, $O_n^+$ and $H_n^{s+}$ are free.
\end{thm}

\section[Examples -- $O_n^+$ and $S_n^+$]{Examples -- $\boldsymbol{O_n^+}$ and $\boldsymbol{S_n^+}$}\label{Sec:Examples}

To illustrate our previous results, we consider three cases, namely those of $O_n^+$, $S_n^+$ and of the quantum ref\/lection groups $H_n^{s+}$ -- which we are going to consider separately in the next section.

We start with the case of $O_n^+$. This free easy quantum group was introduced by S.~Wang in~\cite{QSymmGpFiniteSpWang} and its fusion rules were studied by Banica in~\cite{RepCQGOnBanica}. It appears that it shares the same fusion rules as ${\rm SU}(2)$, i.e., its irreducible representations are denoted by~$u_{k}$ for $k \in \N_0$~-- $u_0$ being the trivial representation~-- and the tensor products decompose into
\begin{gather}\label{Eqn:ClebschGordan}
u_k \otimes u_l = u_{|k-l|} \oplus u_{|k-l|+2} \oplus \cdots \oplus u_{k+l}.
\end{gather}
It is clear from the statement of conditions \ref{Cond:Contra} and \ref{Cond:Entrelac} that they only depend on the fusion rules of the quantum group and not on the quantum group itself. This provides an elementary way to recover the result of Example~\ref{Ex:CP}(a) in this case. In~\cite{FixedPtGabriel}, the case of ${\rm SU}_q(2)$ was studied in detail (see Section~7.1, p.~17, therein). The computation of the $K$-theory of the f\/ixed point algebra also depends only on the fusion rules. Since the fundamental representation of $O_n^+$ corresponds to the natural representation of ${\rm SU}(2)$, the proof of Proposition~7.10 \cite[p.~1031]{FixedPtGabriel} applies \emph{verbatim} and we get Theorem~\ref{Thm:A} below for $\CQG = O_n^+$.

The case of $S_n^+$ is very similar: this CMQG was introduced by S.~Wang in~\cite{FreeProdCQGWang} and its fusion rules were computed by Banica in~\cite{SymmCoactionBanica}. The fusion rules are the same as those of ${\rm SO}(3)$, i.e., we take the fusion rules~\eqref{Eqn:ClebschGordan} of ${\rm SU}(2)$ but consider only \emph{even} representations~$u_{2k}$, $k \in \N_0$. By Example~\ref{Ex:CP}(a), conditions~\ref{Cond:Contra} and~\ref{Cond:Entrelac} are satisf\/ied. The fundamental representation~$u$ of~$S_n^+$ decomposes into $u = u_0 \oplus u_{2} = u_1 \otimes u_1$ and this shows that the proof of Proposition~7.10 of~\cite{FixedPtGabriel} applies again:
\begin{thm}\label{Thm:A}
For $\CQG = O_n^+$ and $\CQG =S_n^+$, the fixed point algebra $\Falg$ obtained from Proposition~{\rm \ref{Prop:Action}} via the fundamental representation~$u$ is a Kirchberg algebra in the UCT class $\UCTclass$ whose $K$-theory is
\begin{gather*}
K_0(\Falg ) = \Z, \qquad K_1(\Falg ) = 0.
\end{gather*}
Moreover, $[1_{\Falg} ]_0 = 1$ and therefore $\Falg$ is $C^*$-isomorphic to the infinite Cuntz algebra~$\CAlg_\infty $.
\end{thm}

\section{Examples -- quantum ref\/lection groups}\label{Sec:QRefGp}

As mentioned in Example~\ref{Ex:WangEtc}(d), the \emph{quantum reflection groups} $H_n^{s+}$ were studied by Ba\-ni\-ca, Belinschi, Capitaine and Collins in~\cite{BBCC}. Their fusion rules were computed by Banica and Vergnioux in their article~\cite{FusionRulesQRefBV}~-- see in particular Theorem~7.3, p.~348, therein. We follow the notations used in their article up to a point: for a reason that will be clear later on, we write~$r_s$ for the representation denoted by~$r_0$ in the original article. For the reader's convenience, we reproduce here the fusion rules of these quantum groups, as described in \cite[Theorem~7.3]{FusionRulesQRefBV}. The monoid $F = \langle \Z/s\Z \rangle $ of words over~$\Z/s\Z$ is equipped with an involution and a fusion ope\-ration:
\begin{enumerate}[label=(\arabic*)]\itemsep=0pt
\item
Involution: $(i_1 \ldots i_k)\,\bar{} = (-i_k) \cdots (-i_1)$.
\item
Fusion: $(i_1 \ldots i_k) \cdot (j_1 \ldots j_l) = i_1 \ldots i_{k-1} (i_k + j_1) j_2 \ldots j_l$.
\end{enumerate}
\newpage

\noindent
Using these relations, the fusion rules can be written
\begin{gather}\label{Eqn:FusionRules}
r_{x} \otimes r_{y} = \sum_{x = v z,\, y = \bar{z} w} r_{v w} + r_{v \cdot w},
\end{gather}
where $v \cdot w$ is not def\/ined when $v$ or $w$ is the empty word. In the case $z= \varnothing$ (``leading terms of the fusion rule''), we distinguish between the term $r_{x y}$ -- that we call the \emph{concatenation term} -- and the term $r_{x \cdot y}$ that we call the \emph{product term}. We are going to compute the $K$-theory of the f\/ixed point algebra generated from the natural representation (indexed by $r_1$).

In the lemma below, we gather useful results:
\begin{lem}
\label{Lem:TotR1}
For any quantum reflection group $\CQG = H_n^{s+}$,
\begin{itemize}\itemsep=0pt
\item
for all $0 <\ell \leqslant s$, the representation $r_{\ell}$ appears in $(r_1)^{\ell}$, i.e., $r_\ell \leqslant (r_1)^\ell$;
\item
moreover, $r_{s-1}$ and $r_1$ are contragredient to one another, i.e., $1 \leqslant (r_1)^s$;
\item
all irreducible representations of $\CQG $ appear as irreducible components of some $r_1^{\ell}$ for $\ell$ large enough. Actually,
\begin{gather}
\label{Eqn:EstimDeg}
r_{\sigma_1 \ldots \sigma_k} \leqslant (r_1)^{\sigma_1 + \cdots + \sigma_k}.
\end{gather}
\end{itemize}
\end{lem}

\begin{Remark}\label{Rk:Cond:Contra}
It follows immediately from this lemma together with Woronowicz's abstract existence of a contragredient for any irreducible representation that condition \ref{Cond:Contra} is satisf\/ied for $H_n^{s+}$ and its representation $r_1$.
\end{Remark}

\begin{proof}[Proof of Lemma~\ref{Lem:TotR1}]
We prove the f\/irst point by induction on $\ell$: the result is true and obvious for $\ell = 1$. For $\ell = 2$,
\begin{itemize}\itemsep=0pt
\item if $s = 2$ then $r_1 \cdot r_1 = r_{11} + r_2 + 1$ and the result is true;
\item if $s > 2$, then $r_1 \cdot r_1 = r_{11} + r_2$, i.e., $r_2 \leqslant (r_1)^2$ and thus the property is true for $\ell = 2$.
\end{itemize}
Let us now assume the property for $\ell >1$, then (provided $\ell +1 < s$), $r_\ell \cdot r_1 = r_{\ell 1} + r_{\ell +1} \leqslant (r_1)^{\ell +1}$, which shows the result for $\ell +1$.

If $\ell + 1 = s$, then the product becomes
\begin{gather*}
r_{\ell} \cdot r_1 = r_{\ell 1} + r_{\ell +1} + 1.
\end{gather*}
The f\/irst two terms above correspond to $z= \varnothing$, while the third one corresponds to $z = \ell$, i.e., $\bar{z} = -\ell = 1$ (equality in $\Z/s\Z$).

Let $r_{\sigma_1 \ldots \sigma_k}$ be any irreducible representation in $H_n^{s+}$ (where all~$\sigma_j$ are taken between~$1$ and~$s$), then it appears as an irreducible component of $r_1^{\sigma_1 + \cdots + \sigma_k}$: our f\/irst point proved that $r_\sigma$ appears in $r_1^{\sigma}$. Therefore, in the decomposition into irreducible components of $r_1^{\sigma_1 + \cdots + \sigma_k} = (r_1)^{\sigma_1} \cdots (r_1)^{\sigma_k}$, there appears a product of~$r_{\sigma_1} \cdot r_{\sigma_2} \cdots r_{\sigma_k}$, which in turn produces a~copy of~$r_{\sigma_1 \ldots \sigma_k}$~-- by using iteratively only the f\/irst (concatenation) term of the fusion rules, for $z = \varnothing$.
\end{proof}

This lemma enables us to set:
\begin{Definition}[degree function for $H_n^{s+}$] Given an irreducible representation $p$ of $\CQG $, we denote by $\delta(p)$ the smallest integer $\ell$ such $p \leqslant r_1^{\ell}$.
\end{Definition}

We can actually give an explicit estimate of $\delta$:
\begin{prop}\label{Prop:DefDeg}
The degree of the irreducible representation $r_{\sigma_1 \ldots \sigma_k}$, where the $\sigma_j$s are chosen with $0 < \sigma_j \leqslant s$, is $\delta(r_{\sigma_1 \ldots \sigma_k}) = \sigma_1 + \cdots + \sigma_k$.
\end{prop}

\begin{proof}Given any irreducible representation $r_{x_1 \ldots x_k}$ with for all $i$, $0 < x_i \leqslant s$, we def\/ine $\sigma(r_{x_1 \ldots x_k}) = x_1 + \cdots + x_k$. It follows by direct examination of \eqref{Eqn:FusionRules} that if $\gamma \leqslant \alpha \cdot \beta$, then
\begin{gather}\label{Eqn:IneqSigma}
\sigma(\gamma) \leqslant \sigma(\alpha) + \sigma(\beta).
\end{gather}
Iterating the argument and combining it with \eqref{Eqn:EstimDeg}, it appears that for all irreducible representation, $\gamma \leqslant (r_1)^{\sigma(\gamma)}$, i.e., $\delta(\gamma) \leqslant \sigma(\gamma)$.

Conversely, if $\gamma \leqslant (r_1)^\ell$, then iterating the ``subadditivity property'' \eqref{Eqn:IneqSigma} of $\sigma$ and using $\sigma(r_1) = 1$, we get: $\sigma(\gamma) \leqslant \ell$. Since this is valid for all $\ell$, we get $\sigma(\gamma) \leqslant \delta(\ell)$. This proves that $\delta(\gamma) = \sigma(\gamma)$ for all irreducible $\gamma$ and concludes the proof.
\end{proof}

\begin{lem} \label{Lem:AllIrred}
For all irreducible representations $\alpha =r_x$, $\beta =r_y$ and $\gamma$ with $\gamma \leqslant \alpha \cdot \beta$,
\begin{gather}\label{Eqn:IneqDeg}
\delta(\gamma) \leqslant \delta(\alpha) + \delta(\beta).
\end{gather}
For $\CQG = H_n^{s+}$, the cases of equality in~\eqref{Eqn:IneqDeg} can only occur for the terms $z = \varnothing$ of~\eqref{Eqn:FusionRules}. For those terms, the equality is true unconditionally for $r_{x y}$, and only if $x_k +y_1 \leqslant s$ for $r_{x \cdot y}$.
\end{lem}

\begin{Remark}
The above lemma is the reason why we use the notation $r_{s}$ instead of $r_0$.
\end{Remark}

\begin{proof}[Proof of Lemma~\ref{Lem:AllIrred}] The inequality \eqref{Eqn:IneqDeg} follows from
\begin{gather*}
\gamma \leqslant \alpha \cdot \beta \leqslant (r_1)^{\delta(\alpha)} (r_1)^{\delta(\beta)} = (r_1)^{\delta(\alpha) + \delta(\beta)}.
\end{gather*}
This is just a variation on the proof of Lemma~\ref{Lem:TotR1} above.

The equality requires to study the behavior of the total degree in the fusion rules, starting from two irreducible representations $\alpha = r_x = r_{x_1 \ldots x_k}$ and $\beta = r_y = r_{y_1 \ldots y_l}$ with f\/inite sequen\-ces~$(x_i)$ and~$(y_j)$ taking their values in $\{1, \ldots , s \}$.

If $\gamma$ arises from a term in~\eqref{Eqn:FusionRules} with $z \neq \varnothing$, then the inequality~\eqref{Eqn:IneqDeg} is strict. Indeed, $\gamma$ could then arise from the $z = \varnothing$ term of $v$ and $w$, where $x = v z$ and $y = \bar{z} w$.

Assuming now that $z = \varnothing$, using the estimate of the degree of Proposition~\ref{Prop:DefDeg}, the term $r_{x_1 \dots x_k y_1 \dots y_l}$ yields an equality case for~\eqref{Eqn:IneqDeg}. The same is true of the term $r_{x_1 \dots (x_k + y_1) \dots y_l}$ provided $x_k + y_1 \leqslant s$. It remains to treat the case of $x_k + y_1 > s$, but such a term corresponds to a strict inequality in~\eqref{Eqn:IneqDeg} and this completes the proof.
\end{proof}

We will use the following notations extensively: let $R_\ell$ (resp.\ $\partial R_\ell$) be the $\Z$-free module constructed on irreducible representations appearing in $(r_1)^\ell$ (resp.\ appearing in $(r_1)^\ell$ and not in any $(r_1)^k$ for $0 \leqslant k < \ell$).

It follows immediately from the def\/inition of $R_\ell$ that $R_\ell \cdot R_k \subseteq R_{\ell+k}$, where the product $R_\ell \cdot R_k$ is taking place in the fusion ring of $\CQG $.

\begin{lem}The fusion rules \eqref{Eqn:FusionRules} actually ensure that in $\Z/s \Z$:
\begin{gather}\label{Eqn:EgDeg}
[\delta(\gamma)] = [\delta(\alpha)] + [\delta(\beta)].
\end{gather}
\end{lem}

\begin{proof}
Direct examination and Proposition~\ref{Prop:DefDeg} show that if $z = \varnothing$, then $r_{xy} = r_{i_1 \ldots i_k j_1 \ldots j_l}$ has degree $\delta(r_{xy}) = i_1 + \cdots + i_k + j_1 + \cdots + j_l = \delta(r_x) + \delta(r_y)$. The degree of $r_{x \cdot y}$ is the same in $\Z/s \Z$ (since a simplif\/ication in~$\Z/s\Z$ may happen in the fusion $x \cdot y$).

More generally, if $z \neq \varnothing$, then the def\/inition of the involution on the monoid $F$ ensures that taking out both $z$ and $\bar{z}$ do not change~$[ \delta(r_{v w})]$ in~$\Z/s \Z$. However, simplifying by $z$ and $\bar{z}$ lessen the total degree of the expression, i.e., the degree~$\delta(r_{vw})$ has to be strictly less than $\delta(r_x) + \delta(r_y)$. The same argument applies to~$r_{v+w}$.
\end{proof}

We are now in position to compute the \emph{chain group} $\ChainGp(\CQG )$ as introduced in \cite{HilbertCSystBaumgLledo,CenterCpctGpMueger}. This object is also known as \emph{universal grading group}, see, e.g.,~\cite{NilpotFusionCatGN}.
\begin{prop}\label{Prop:IddChainGp}
The chain group of $\CQG = H_n^{s+}$ is $\ChainGp(\CQG ) =\Z/s \Z$.
\end{prop}

\begin{proof}
The equation \eqref{Eqn:EgDeg} shows that if $\tau_1, \ldots , \tau_k, p, q$ are irreducible representations and that they satisfy $p, q \leqslant \tau_1 \cdot \tau_2 \cdots \tau_k$ (both $p$ and $q$ appear in the fusion product), then (in the group~$\Z/s\Z$)
\begin{gather*}
[\delta(p)] = [\delta(\tau_1)] + \cdots + [\delta(\tau_k)] = [\delta(q)].
\end{gather*}
In other words, if $p$ and $q$ have the same class in the chain group $\ChainGp(\CQG )$, then $[\delta(p)] = [\delta(q)]$.

Conversely, take $[\delta(p)] = [\delta(q)]$ in $\Z/s \Z$. If $\delta(p) = \delta(q)=\ell$, then both representations appear in $(r_1)^\ell$ and they have the same class in the chain group. Otherwise, without loss of generality, we can assume that $\delta(p) > \delta(q)$ and thus there is an integer $k$ such that $\delta(p) = \delta(q) + k s$. By def\/inition of $\delta(q)$, $q \leqslant (r_1)^{\delta(q)}$. We can then use Lemma~\ref{Lem:TotR1} (and especially the part $1 \leqslant (r_1)^s$) to show
\begin{gather*}
q = q \cdot 1^k \leqslant (r_1)^{\delta(q)} (r_1)^s \cdots (r_1)^s = (r_1)^{\delta(q) + k s}.
\end{gather*}
This in turn proves that $p$ and $q$ share the same class in~$\ChainGp(\CQG )$. The equation~\eqref{Eqn:EgDeg} then ensures that the group law in $\ChainGp(\CQG )$ and~$\Z/s\Z$ coincide.
\end{proof}

\begin{Remark}\label{Rk:Cond:Entrelac}
It follows from this evaluation of the chain group $\ChainGp(\CQG )$ together with the pro\-per\-ty $1 \leqslant (r_1)^s$ of Lemma~\ref{Lem:TotR1} that condition \ref{Cond:Entrelac} is satisf\/ied for $\CQG =H^{s+}_n$ equipped with its representation~$r_1$, for the integers $k_0 = s$ and $N = 1$.
\end{Remark}

\begin{cor}There is a decomposition
\begin{gather}\label{Eqn:DecompRell}
R_{\ell} = \bigoplus_{0 \leqslant k \leqslant \ell, [k] = [\ell]} \partial R_k,
\end{gather}
where the equality $[k] = [\ell]$ takes place in $\Z/s\Z$.
\end{cor}

\begin{proof}This is a consequence of $p \leqslant (r_1)^\ell \implies [\delta(p)] = [\ell]$ in $\Z/s\Z$ together with $1 \leqslant (r_1)^{s}$.
\end{proof}

Finally, we will need the notion of \emph{length} $\lambda(r_s)$ of an irreducible representation $r_s$, which is just the length (number of letters) of its indexing sequence $s$. It is clear from~\eqref{Eqn:IneqDeg} that for all $\gamma \leqslant \alpha \cdot \beta$, $\lambda(\gamma) \leqslant \lambda(\alpha) + \lambda(\beta)$ with equality \emph{only} for the concatenation term of $z = \varnothing$.

\begin{thm}\label{Thm:Kth}
For $\CQG = H^{s+}_n$ and its representation $\alpha = r_1$, condition~{\rm \ref{Cond:Entrelac}} is satisfied for the integer $k_0 = s$ and the computation of $K$-theory yields
\begin{gather*}
K_0(\CAlg^\alpha) = \bigoplus_{\N} \Z, \qquad K_1(\CAlg^\alpha) = 0.
\end{gather*}
\end{thm}

The proof of this theorem is going to require a few intermediate lemmas and a restatement of the problem.

Indeed, Theorem~\ref{Thm:Kth} above is stated in terms of the representation $\alpha= r_1$, but the computations below will be easier if we consider the case of $\alpha = r_1^s$~-- which yields isomorphic results, according to Proposition~\ref{Prop:IddChainGp} above and \cite[Proposition~7.8, p.~1029]{FixedPtGabriel}.

Consider the maps $\varphi \colon R \to R$ and $\psi \colon R \to R$ given on all $a \in R$ by
\begin{gather*}
\varphi(a) = a (r_1^s - 1), \qquad \psi(a)= a r_1^s.
\end{gather*}
The previous properties of degree show that these maps induce $\varphi_\ell \colon R_\ell \to R_{\ell+s}$ and $\psi_\ell \colon R_\ell \to R_{\ell+s}$. The $K$-theory of the f\/ixed point algebra $\InvJ^\alpha$ is given by the inductive limit of the system
\begin{gather*}
\cdots \to R_{\ell} \xrightarrow{\psi_\ell} R_{\ell+s} \xrightarrow{\psi_{\ell+s}} R_{\ell + 2s} \to \cdots.
\end{gather*}
The general theory presented in~\cite{FixedPtGabriel} (see in particular Theorem~5.4, p.~1025) shows that $K_0(\CAlg^\alpha)$ is obtained as the cokernel of the map $\varphi \colon \lim\limits_{\to } R_\ell \to \lim\limits_{\to } R_\ell$ def\/ined from the system
\begin{gather}\label{Eqn:SysPhi}
\begin{split}&
\xymatrix@R=8mm{%
\cdots\ar[r] &R_{\ell} \ar[r]^{\psi_\ell} & R_{\ell+s} \ar[r]^{\psi_{\ell+s}} & R_{\ell+2s} \ar[r]^{\psi_{\ell+2s}} & R_{\ell+3s} \ar[r]^{\psi_{\ell+3s}} & R_{\ell+4s} \ar[r]& \cdots \\
\cdots\ar[r] &R_{\ell} \ar[r]_{\psi_\ell} \ar[ur]^-{\varphi_\ell} & R_{\ell+s} \ar[r]_{\psi_{\ell+s}} \ar[ur]^-{\varphi_{\ell+s}} & R_{\ell+2s} \ar[r]_{\psi_{\ell+2s}} \ar[ur]^-{\varphi_{\ell+2 s}} & R_{\ell+3s} \ar[r]_{\psi_{\ell+3 s}} \ar[ur]^-{\varphi_{\ell+3 s}} & R_{\ell+4s} \ar[r]& \cdots. \\
}
\end{split}
\end{gather}
All the squares in the diagram above are commutative -- indeed, it amounts to proving that for any $a \in R$, $a r_1^s (r_1^s - 1) = a (r_1^s - 1) r_1^s$. Thus, the map $\varphi \colon \lim\limits_{\to } R_\ell \to \lim\limits_{\to } R_\ell$ is well-def\/ined and we can compute its cokernel. To this end, we start by computing the cokernels at each f\/inite level~$\ell$ and it is a well-known property that we will obtain the overall cokernel as inductive limit of those f\/inite cokernels.

The evaluation below is the corner stone of our argument:
\begin{lem}\label{Lem:RelRs}
For any irreducible representation $r_x$ with $\delta(r_x) = \ell$, we have
\begin{gather}\label{Eqn:RelRs}
\varphi(r_x) = r_x (r_1^s- 1) = r_{x\text{\scriptsize $\underbrace{1\dots 1}_{s \, \text{terms}}$}} + r_{x s} + m,
\end{gather}
where $m$ is a $\Z$-linear combination of irreducible representations in $R_{\ell + s}$ which \emph{do not} contain any term $r_{x1\dots 1}$ and $r_{x s}$.
\end{lem}

\begin{proof}It is clear from the fusion rules that by taking only the concatenation term for $z = \varnothing$ in the $s$ successive fusion products of $r_x$ with $r_1$, we obtain a term $r_\mu = r_{x\text{\scriptsize $\underbrace{1\dots 1}_{s \, \text{terms}}$}}$. Moreover, this irreducible irrepresentation has maximal length (namely $\lambda(r_\mu) = \lambda(r_x) + s$) among those appearing in the product $r_x (r_1^s- 1)$. Given $\gamma \leqslant \alpha \cdot \beta$, we know that $\lambda(\gamma) \leqslant \lambda(\alpha) + \lambda(\beta)$ is actually an \emph{equality} only for the concatenation term of $z = \varnothing$. It follows that there is only one way to obtain a representation of such length. Thus, no further term involving $r_{\mu}$ appear in~$\varphi(r_x)$.

For $r_{x s}$, the argument is slightly dif\/ferent: f\/irst, we remark that it has maximum degree ($\delta(r_{xs}) = \delta(r_x) + s$). This implies that it was obtained by taking only terms with $z = \varnothing$ in the successive fusion products. We then remark that its length is minimal among those terms obtained by taking only \emph{leading terms} ($z = \varnothing$) in the fusion. This in turn ensures that it is (and can only be) obtained by taking \emph{product terms} in the $s$ successive fusion products. Thus, no further term involving $r_{x s}$ appear in~$\varphi(r_x)$.
\end{proof}

\begin{lem}\label{Lem:DirectSum}
For all $\ell$, there is free $\Z$-module $C_{\ell} \subseteq R_\ell$ such that
\begin{gather}\label{Eqn:DirectSum}
R_{\ell+s} = C_{\ell +s} \oplus \varphi_{\ell}(R_\ell).
\end{gather}
Moreover, this free $\Z$-module decomposes according to the degree into
\begin{gather}\label{Eqn:DecompCell}
C_{\ell} = \bigoplus_{0 \leqslant k \leqslant \ell,\, [k] = [\ell]} \partial C_{k}.
\end{gather}
\end{lem}

{\samepage\begin{Remark}\label{Rk:Coker}
The statement above calls for several remarks:
\begin{itemize}\itemsep=0pt
\item The notation $R_{\ell+s} = C_{\ell +s} \oplus \varphi_{\ell}(R_\ell)$ indicates that any element of $R_{\ell+s}$ can be written \emph{in a~unique way} as a sum of an element of $C_{\ell+s}$ and an element of $\varphi_{\ell}(R_\ell)$.
\item An immediate consequence of \eqref{Eqn:DirectSum} is that we can thus identify $C_{\ell+s}$ with the cokernel $R_{\ell+s}/\varphi(R_{\ell})$.
\item In the decomposition \eqref{Eqn:DecompCell}, we use obvious notations similar to those of \eqref{Eqn:DecompRell}.
\end{itemize}
\end{Remark}}

\begin{proof}[Proof of Lemma~\ref{Lem:DirectSum}] We proceed by induction: for a minimal level $0 \leqslant \ell < 2 s$, the decomposition~\eqref{Eqn:DirectSum} shows that we just have to f\/ind $C_\ell$ s.t.\ $R_\ell = C_\ell \oplus \varphi_{\ell -s}(R_{\ell -s})$. Given any $a \in R_\ell$, we can use relation~\eqref{Eqn:RelRs} to cancel any term of the form $r_{x\text{\scriptsize $\underbrace{1\dots 1}_{s \, \text{terms}}$}}$ appearing in $a$. If we then def\/i\-ne~$C_{\ell}$ as the free $\Z$-module generated on all irreducible representations appearing in~$R_\ell$ which are \emph{not} of the form $r_{x\text{\scriptsize $\underbrace{1\dots 1}_{s \, \text{terms}}$}}$, then clearly
\begin{gather*}
R_\ell = C_\ell \oplus \varphi_{\ell -s}(R_{\ell -s}).
\end{gather*}
Let us now assume that at level $\ell$, we have a decomposition
\begin{gather*}
R_\ell = C_\ell \oplus \varphi(R_\ell),
\end{gather*}
where $C_{\ell} = \bigoplus_{0 \leqslant k \leqslant \ell, [k] = [\ell]} \partial C_{k}$, we want to prove that $R_{\ell+s}$ admits a similar decomposition.

A consequence of the decomposition \eqref{Eqn:DecompRell}, is that $\varphi(R_{\ell}) = \bigoplus_{0 \leqslant k \leqslant \ell, [k] = [\ell]} \varphi(\partial R_k)$.

Let us now introduce $\partial C_{\ell+2}$ as the free $\Z$-module generated by all irreducible representation of degree \emph{exactly} $\ell+s$ which are \emph{not} of the form $r_{x\text{\scriptsize $\underbrace{1\dots 1}_{s \, \text{terms}}$}}$, then
\begin{gather*}
R_{\ell+s} = \partial C_{\ell+s} \oplus \varphi( \partial R_{\ell}) \oplus R_\ell = \partial C_{\ell+s} \oplus \bigoplus_{0 \leqslant k \leqslant \ell, [k] = [\ell]} \partial C_{k} \oplus \varphi(R_{\ell+s}).
\end{gather*}
This completes the proof of the existence of the $\Z$-free module $C_{\ell} = \bigoplus_{0 \leqslant k \leqslant \ell, [k] = [\ell]} \partial C_{k}$ which implements the cokernel in~$R_\ell$.
\end{proof}

Let us now study the connecting maps between these cokernels. Remember from the commutation relations appearing in \eqref{Eqn:SysPhi} that all connecting maps $\psi_{\ell} \colon R_\ell \to R_{\ell +s}$ induce quotient maps at the level of cokernels, which we denote by
\begin{gather*}
\widetilde{\psi_{\ell}} \colon \ R_\ell / \varphi(R_{\ell-s}) \to R_{\ell +s}/ \varphi(R_\ell).
\end{gather*}
\begin{lem}\label{Lem:ConnectingMaps}
The connecting maps between the cokernels are the identity: for any $a \in R_\ell$, $\widetilde{\psi_{\ell}}([a]_\ell) = [a]_{\ell+s}$,
where $[a]_\ell$ and $[a]_{\ell +s}$ are the class of $a$ in $R_\ell / \varphi(R_{\ell - s})$ and $R_{\ell+s} / \varphi(R_\ell)$, respectively.
\end{lem}

\begin{proof}Indeed, take any $a \in R_\ell$, then $\psi(a) = \varphi(a) + a$. We know that in the cokernel $R_{\ell+s} / \varphi(R_\ell)$, $[\varphi(a)]_{\ell+s} = 0$, thus $[\psi(a)]_{\ell+s} = [a]_{\ell+s}$.
\end{proof}

\begin{Remark}A consequence of the above Lemma~\ref{Lem:ConnectingMaps} is that the inductive limit $\lim\limits_{\to } C_\ell$ is simply the increasing union of free $\Z$-modules and it suf\/f\/ices to estimate the number of irreducible representations of~$H_n^{s+}$ of degree $\ell +s$ which are not of the form $r_{x\text{\scriptsize $\underbrace{1\dots 1}_{s \, \text{terms}}$}}$. We do precisely this in the next lemma.
\end{Remark}

\begin{lem}\label{Lem:EvalDeg}Let $m_{\ell+s}$ be the number of irreducible representations of $H_n^{s+}$ of degree $\ell +s$ which are not of the form $r_{x\text{\scriptsize $\underbrace{1\dots 1}_{s \, \text{terms}}$}}$, then $m_{\ell+s} \to \infty $.
\end{lem}

\begin{proof}
Let us introduce the number $n_\ell$ of irreducible representations of degree \emph{exactly} $\ell$, then relation~\eqref{Eqn:RelRs} ensures that
\begin{gather}\label{Eqn:nEll}
n_{\ell+s} \geqslant 2 n_{\ell}.
\end{gather}
Indeed, for each irreducible representation $r_x$ of degree $\ell$, there are at least two irreducible representations of degree $\ell +2$, namely $r_{x\text{\scriptsize $\underbrace{1\dots 1}_{s \, \text{terms}}$}}$ and $r_{xs}$. This also forces $m_{\ell+s} \geqslant n_{\ell}$. The rela\-tion~\eqref{Eqn:nEll}, together with the equality $n_0 = 1$, shows that~$n_\ell$ (and therefore~$m_{\ell}$) tends to inf\/inity when $\ell$ tends to inf\/inity.
\end{proof}

\begin{proof}[Proof of Theorem~\ref{Thm:Kth}] It follows from Example~\ref{Ex:CP}(a) that conditions~\ref{Cond:Contra} and~\ref{Cond:Entrelac} are satisf\/ied for the fusion rules of the quantum ref\/lection group $\CQG = H_n^{s+}$. Consequently, Proposition~\ref{Prop:PrevResults} applies to the f\/ixed point algebra~-- which is therefore determined up to $*$-isomorphism by its $K$-theory.

To compute $K_*(\CAlg^\alpha)$, we use the inductive system~\eqref{Eqn:SysPhi}. We f\/irst evaluate $K_1(\CAlg^\alpha)$: according to \cite[Theorem~5.4, p.~1025]{FixedPtGabriel} this $K$-group is the kernel of the map $\varphi$ def\/ined by the inductive system~\eqref{Eqn:SysPhi}. If $c$ is a nonzero element in~$R$, it can be realised on a f\/inite level $\ell$. Let us consider the top length nonvanishing irreducible representations appearing in $c \in R_\ell$ and write $c = \sum\alpha_j r_{x^j_1 \dots x^j_\lambda} + m$ where $\alpha_j \in \Z \setminus \{ 0\}$, $\lambda$ is the maximum length of irreducible representations in~$c$ and~$m$ is a combination of irreducible representations with lower length. Following Lemma~\ref{Lem:RelRs},
\begin{gather*}
\varphi(c) = \sum_{} \alpha_j r_{x^j_1 \dots x^j_\lambda 1 \dots 1} + m',
\end{gather*}
where $m'$ is a linear combination of irreducible representations which \emph{do not} contain any term $ r_{x^j_1 \dots x^j_\lambda 1 \dots 1}$ (i.e., maximal length terms). It follows that $\varphi(c) - c \neq 0$ (since irreducible representations in $c$ have length at most $\lambda$ and $r_{x^j_1 \dots x^j_\lambda 1 \dots 1}$ has length $\lambda + s$, no cancellation can occur). Essentially the same argument proves that if $c \in R_\ell$ is nonvanishing, then $\psi_\ell(c) \in R_{\ell +s}$ is also nonvanishing. Consequently, $\ker(1 - \varphi_*) = K_1(\CAlg^\alpha) = \{ 0 \}$.

The computation of $K_0(\CAlg^\alpha)$ is an easy consequence of Lemma~\ref{Lem:DirectSum}, Remark~\ref{Rk:Coker} and Lem\-mas~\ref{Lem:ConnectingMaps} and~\ref{Lem:EvalDeg}. The proof of Theorem~\ref{Thm:Kth} is thus complete.
\end{proof}

As a f\/inal comment, this paper shows how techniques from classif\/ication theory for $C^*$-algebras and a thorough understanding of fusion rules can be combined to identify free actions of compact quantum groups on $C^*$-algebras. The characterisation of the f\/ixed point algebra requires a concrete computation of $K$-theory, and explains why we restricted ourselves to examples in the second part of the paper. Similar results should however be possible for other classes of CQGs, as soon as we have a f\/ine comprehension of their fusion rules.

\subsection*{Acknowledgements}

The second author was partially funded by the ERC Advanced Grant on Non-Commutative Distributions in Free Probability, held by Roland Speicher, Saarland University. The f\/irst author was supported by the Danish National Research Foundation through the Centre for Symmetry and Deformation (DNRF92) and by the Engineering and Physical Sciences Research Council Grant EP/L013916/1, since the f\/irst results of this work were obtained during the f\/irst author's postdoc in Glasgow.

Both authors are grateful to Roland Speicher's ERC Advanced Grant and Christian Voigt for enabling their respective stays in Scotland where this collaboration started. They also thank the anonymous referees for their thourough reviews and remarks.

\pdfbookmark[1]{References}{ref}
\LastPageEnding

\end{document}